\documentclass[a4paper,11pt,reqno]{amsart} 
\pdfoutput=1

\usepackage[top=3.5cm,bottom=3cm,left=2.6cm,right=2.6cm]{geometry}

\usepackage[T1]{fontenc}
\usepackage[utf8]{inputenc}

\usepackage{lmodern}

\usepackage{microtype} 

\usepackage{url} 

\usepackage[version=3]{mhchem} 

\usepackage{amsmath,amssymb,thmtools,mathtools}

\usepackage{graphicx}
\usepackage{caption}
\usepackage{subcaption}
\usepackage{float}
\usepackage{wrapfig}
\usepackage{tikz}
\usepackage{tikz-3dplot}
\usepackage{pgfplots}
\pgfplotsset{compat=1.18}

\usetikzlibrary{intersections}
\usetikzlibrary{arrows.meta, calc}

\usepackage{multirow}
\usepackage{array} 

\usepackage{mathrsfs} 

\usepackage{bm,bbm}

\usepackage{bbm} 

\usepackage{enumitem}

\usepackage[perpage]{footmisc}

\usepackage{tikz-cd}

\usepackage[all]{xy}

\usepackage[textsize=footnotesize, 
  linecolor=gray, bordercolor=gray, backgroundcolor=gray!25, 
  textcolor=black]{todonotes}

\usepackage[unicode,psdextra,pdfusetitle,
bookmarks, bookmarksdepth=2, colorlinks=true, linkcolor=blue!50!black, citecolor=black!70, urlcolor=blue!50!black,hypertexnames=false]{hyperref}

\usepackage[sort,compress,numbers]{natbib}

\usepackage[indent,skip=.2\baselineskip]{parskip}
\usepackage{comment}

\usepackage{amsthm}
\usepackage[nameinlink,capitalize]{cleveref}
 
\newtheorem{theorem}{Theorem}[section]
\newtheorem{lemma}[theorem]{Lemma}
\newtheorem{proposition}[theorem]{Proposition}
\newtheorem{corollary}[theorem]{Corollary}
 
\theoremstyle{definition}
\newtheorem{definition}[theorem]{Definition}
\newtheorem{example}[theorem]{Example}

\newtheorem{remark}[theorem]{Remark}

\newtheorem{theoremalphabetic}{Theorem}

\Crefname{ex}{Example}{Examples}
\Crefname{thm}{Theorem}{Theorems}
\Crefname{lem}{Lemma}{Lemmas}
\Crefname{prop}{Proposition}{Propositions}
\Crefname{cor}{Corollary}{Corollaries}
\Crefname{con}{Conjecture}{Conjectures}
\Crefname{def}{Definition}{Definitions}
\Crefname{alg}{Algorithm}{Algorithms}
\Crefname{rmk}{Remark}{Remarks}
\Crefname{thmalph}{Theorem}{Theorems}
\Crefname{equation}{}{}

\numberwithin{equation}{section}
\numberwithin{figure}{section}

\DeclareMathOperator{\Zar}{Zar} 
\DeclareMathOperator{\rank}{rank} 
\DeclareMathOperator{\im}{im} 
\DeclareMathOperator{\conv}{conv} 
\DeclareMathOperator{\diag}{diag} 
\DeclareMathOperator{\Aff}{Aff} 
\DeclareMathOperator{\sign}{sign} 
\DeclareMathOperator{\intt}{int} 
\DeclareMathOperator{\vertices}{vert} 
\DeclareMathOperator{\relint}{relint}
\DeclareMathOperator{\Sing}{Sing}
\DeclareMathOperator{\Hess}{Hess}

\newcommand{\N}{\ensuremath{\mathbb{N}}}

\newcommand{\Z}{\ensuremath{\mathbb{Z}}}

\newcommand{\R}{\ensuremath{\mathbb{R}}}
\newcommand{\C}{\ensuremath{\mathbb{C}}}

\newcommand{\A}{\ensuremath{\mathcal{A}}}
\newcommand{\mR}{\ensuremath{\mathcal{R}}}

\DeclareTextFontCommand{\bfemph}{\bfseries\em}
\newcommand{\term}{\bfemph}

\renewcommand{\k}{\kappa}

\makeatletter
\renewcommand*\env@matrix[1][*\c@MaxMatrixCols c]{%
  \hskip -\arraycolsep
  \let\@ifnextchar\new@ifnextchar
  \array{#1}}
\makeatother

\title[Copositivity, discriminants and nonseparable signed supports]{Copositivity, discriminants and\\ nonseparable signed supports}
\author{Elisenda Feliu, Joan Ferrer, and Máté L. Telek}
\date{\today}

\begin{document}
 \vspace*{-2em}
\maketitle
 \vspace{-2em}

\begin{abstract}
We establish a connection between discriminants and copositivity of sparse   Laurent polynomials. More generally, we consider signomials, which are defined on the positive orthant by  linear combinations of monomials with  real exponents, and their signed support, consisting of the support of the signomial and the sign of the coefficients. 
We provide a criterion to determine whether a given signomial is copositive, that is, only attains nonnegative values, which relies on finding the first intersection point in coefficient space of 
a path and a signed discriminant: the path contains the coefficient vector of the signomial and preserves its signs, and the signed discriminant encodes the singular  signomials with the given signed support. 
If the signed support satisfies a combinatorial condition termed nonseparability, we additionally show  that this intersection consists of one point, and that tracking one path with homotopy continuation suffices to decide upon copositivity.  

Building on these results, we show that copositive polynomials with nonseparable signed support can always be decomposed into a sum of nonnegative circuit polynomials,  generalising  thereby previously known supports having this property.  
\end{abstract}

\section{Introduction}

The cone $\mathcal{P}_{n, 2\delta}$ of globally nonnegative $n$-variate homogeneous real polynomials of degree~$2\delta$ is a central object in real algebraic geometry with applications in several areas, notably in polynomial optimization \cite{Parrilo,Theobald}. Providing algebraic descriptions of this cone and developing computationally tractable methods to determine nonnegativity are major tasks in the field. As deciding membership in $\mathcal{P}_{n,2\delta}$ is NP-hard \cite{Nesterov}, one often tests membership via a proper subcone using \emph{certificates} of nonnegativity—that is, decompositions into sums of manifestly nonnegative polynomials. The cone $\Sigma_{n,2\delta}$ of polynomials that are sums of squares   is the most extensively studied, as containment can be verified using semidefinite programming \cite{Lasserre_Global_optim,Parrilo}.  In 1888, Hilbert showed that $\Sigma_{n,2\delta}=\mathcal{P}_{n,2\delta}$ if and only if $n=2$, or $\delta=1$, or   $(3,2)$
\cite{Hilbert_SOS}. This motivated Hilbert's 17th problem \cite{hilbert}, which became a major driving force in the development of real algebraic geometry during the $20$th century. 
Other cones used to certify nonnegativity are the cone of polynomials that are 
Sums Of Nonnegative Circuits (SONC-cone) \cite{wolff_nonneg}, or that of the polynomials that are Sums of Arithmetic–Geometric Exponentials (SAGE) \cite{chandrasekaran_SAGE,framework_SAGE_SONC}. Independent works \cite{murray_Newtonpol,Wang_Nonneg} established an equivalence: a polynomial admits a SONC certificate if and only if it admits a SAGE certificate.

Recently, extensive research has focused not only on global nonnegativity but also on nonnegativity on semialgebraic sets. Among these, the positive real orthant $\R^n_{>0}$ is arguably the most prominent, and arises as the natural context for the SONC- and SAGE-cones. When restricting to $\R^n_{>0}$, we can consider \emph{signomials}, where the exponents of the monomials are allowed to take real, and not only integer, values. 
Following  the terminology in \cite{motzkin_copositive,Moment_polyomial_optimization}, we will say that a signomial is \emph{copositive} if it   attains nonnegative values on $\R^n_{>0}$.
The problem of deciding upon copositivity has attracted growing attention, driven by applications in graph theory, particle physics, and reaction network theory \cite{deKlerkPasechnik, copositive_feynman, Multistationarity2017}.   

In the context of copositivity, it is more convenient to fix the set of monomials of the signomial rather than fixing the degree. This is the \textit{sparse polynomial/signomial} setting, which dates back, at least, to the work of Gel'fand, Kapranov and Zelenvinsky \cite{gelfand}. Additionally, as the domain of a signomial is $\R^n_{>0}$, the sign of each term of the signomial is determined by the sign of its coefficient, and one may focus on the set of all signomials with given support and predetermined sign for each coefficient. This leads us to consider signomials written as 
\[f_c=\sum_{a\in\A_+} c_a\,x^a  - \sum_{b\in\A_-} c_b\,x^b \,\]
for $c\in \R^{\A_+\cup \A_-}_{>0}$. The pair $(\A_+,\A_-)$ is called the \emph{signed support}  of $f_c$.

One of the main goals of this work is to investigate the relation between the boundary of the cone of \emph{sparse} copositive signomials, discriminants, the signed support $(\A_+,\A_-)$, and SONC certificates.
To explain this, given a pair $(\A_+,\A_-)$, we let $\A=\A_+\cup \A_-$ and consider the \textit{signed $(\A_+,\A_-)$-discriminant}, consisting of all $c\in \R^{\A}_{>0}$ for which $f_c$ or 
a truncation 
\[f_c^\Gamma \coloneqq
\sum_{a\in\A_+\cap \Gamma} c_a\,x^a  - \sum_{b\in\A_-\cap \Gamma} c_b\,x^b\,,\quad \text{$\Gamma$ nonempty face of $\conv(\A)$}\,,\]
has a positive singular zero. 
This discriminant can be seen to arise by fixing the signs of the coefficients in the positive $\A$-discriminant of \cite{bihan2022new}, which decomposes the space of coefficients in a disjoint union of connected sets for which the set of  zeros of $f_c$ has constant isotopy type. 

If the signed support of a  signomial $f_c$ satisfies $\A_-=\varnothing$, then $f_c$ is clearly copositive. If some element of $\A_-$ is a vertex of $\conv(\A)$, then it is well-known that   $f_c$ is not copositive. 
Our first main result is a criterion to decide whether $f_c$ is a copositive signomial in the remaining cases, for which deciding upon copositivity is not immediate. 
A simplified version of this result is  as follows.  

\begin{theoremalphabetic}[\Cref{thm:criterion:copositivity}]\label{alphabetic:thm:criterion}
Let $f_c$ be a signomial with signed support $(\A_+,\A_-)$ satisfying  $\A_-\neq \varnothing$ and    $\A_-\subseteq \conv(\A_+)$. 
Let $c(t)\subseteq\R^\A_{>0}$ be the path defined for $t\in \R_{>0}$ by $c(t)_a=c_a$ if $a\in \A_+$ and $c(t)_a=c_a\,t$ if $a\in \A_-$.   Then, the set
    \[T\coloneqq\bigcup_{\substack{\varnothing\neq \Gamma \text{ face of}\\ \conv(\A)}}\Big\{t\in\R_{>0}\colon f_{c(t)}^\Gamma\text{ has a positive singular zero}\Big\}\]
   has a minimum $t_*$, and $f_c$ is copositive if and only if $t_*\geq 1$.
\end{theoremalphabetic}

Observe that in the definition of the set $T$ we also consider $\conv(\A)$ to be a face. 
An interpretation of \Cref{alphabetic:thm:criterion} is that the complement of the signed discriminant has 
a unique connected component consisting of coefficients for which the zero set of the signomial is empty. The set $T$ in the theorem gives us the 
values of $t$ at which the 
path $c(t)$ intersects the signed $(\A_+,\A_-)$-discriminant. Computing it involves finding the positive zeros of finitely many signomial systems, specifically  of the   critical systems $F^\Gamma(c(t),x)$ whose zeros are the pairs $(t,x)$ for which $x$ is a positive singular zero of $f^\Gamma_{c(t)}$. 
At a first glance, this might seem similar to deciding copositivity by computing the critical points of $f_c$ in $\R^n_{>0}$ and evaluating $f_c$ at each of them. However,  \Cref{alphabetic:thm:criterion} avoids the evaluation step, which is a major source of numerical error. Instead, deciding whether $t_*\geq 1$ can be addressed using certification  of zeros of polynomial systems.

\smallskip
We next focus on the problem of deciding whether a copositive signomial admits a SONC decomposition, that is,  can be written as a sum of nonnegative circuits. 
An important open question for this problem is to characterise the sign supports $(\A_+,\A_-)$ for which any copositive signomial having this signed support  admits a SONC decomposition. 
Several geometric conditions on $(\A_+,\A_-)$ ensuring this property have already been described \cite{murray_Newtonpol,wolff_nonneg,Wang_Nonneg}. The pairs satisfying these conditions are particular instances of nonseparable signed supports, a concept that we learnt from the  upcoming work ~\cite{Mate&Timo2025}.  The pair $(\A_+,\A_-)$ is said to be \textit{nonseparable} if $\A_-\subseteq \conv(\A_+)$ and for each regular subdivision of the points in $\A_+$, there exists a cell $D$ of the subdivision containing $\A_-$. Nonseparability allows points in $\A_-$ to be at the boundary of both $D$ and $\conv(\A_+)$.

Our second main result states that given a nonseparable signed support, the SONC-cone and the cone of copositive signomials agree. 

\begin{theoremalphabetic}[\Cref{thm:nonseparable:signed:supports:SONC}]\label{alphabetic:thm:sonc}
If $f$ is a signomial with nonseparable signed support such that $\A_-$ is in the interior of $\conv(\A_+)$, then  $f$ is copositive if and only if it admits a SONC decomposition. 
\end{theoremalphabetic}

To show this result, we first prove that any signomial that has a positive singular zero and nonseparable signed support admits a SONC decomposition. To this end, we build and extend  several ideas from  \cite{Wang_Nonneg}. With this in place, we use our criterion in \Cref{alphabetic:thm:criterion} to show that the SONC decomposition exists for any copositive signomial, also in the absence of singular zeros.

A consequence of \Cref{alphabetic:thm:criterion} and \Cref{alphabetic:thm:sonc}
is that for Laurent polynomials, the boundary of the cone of copositive polynomials with nonseparable signed support $(\A_+,\A_-)$ is contained in the signed $(\A_+,\A_-)$-discriminant (\Cref{thm:boundary}). This aligns well with other work in the literature connecting discriminants and the boundary of the cone of copositive polynomials. For example,  
in \cite[Proposition 6.5]{Nie_discriminants} it is shown that the algebraic boundary of the    cone of copositive homogeneous polynomials of degree $\delta$ is the principal $\A$-determinant, where $\A = \{ a \in \N^n : a_1 + \dots + a_n = \delta \}$.

Our third and final main result addresses the computation of the set $T$ in \Cref{alphabetic:thm:criterion} for nonseparable signed supports. 

\begin{theoremalphabetic}[\Cref{thm:1solution,thm:Jacobian:nondeg:tracking}]
\label{alphabetic:thm:1sol} 
Let $f_c$ be a signomial with nonseparable signed support $(\A_+,\A_-)$ such that $\conv(\A_+)$ is full dimensional and contains $\A_-$ in its interior. 

Then the set $T$  in \Cref{alphabetic:thm:criterion} has exactly one element $t_*$. Additionally, the signomial $f_{c(t_*)}$ has exactly one positive singular zero $x_*$ and $(t_*,x_*)$ is a nondegenerate  zero of the critical system. 
\end{theoremalphabetic}

The critical systems of nonseparable signed supports are thus computationally well behaved and 
\Cref{alphabetic:thm:1sol} allows for an efficient practical implementation of \Cref{alphabetic:thm:criterion} using homotopy continuation methods and tracking one solution path with parameter homotopies. 
In the general case, if the zero sets of the critical systems defining the set $T$ in \Cref{alphabetic:thm:criterion} have positive dimension, then numerical implementations of the criterion for copositivity might miss solutions, possibly leading to a wrong outcome of the criterion (see \Cref{subsec:limitations}). We show in \Cref{prop:generic:coefficients:isolated} that this behaviour does not happen for generic coefficients. 

On a closing note,  the ideas developed in this paper have the potential to lead to  new closed formulas on copositivity depending on the coefficients of a signomial, similar to the well-established circuit numbers. To illustrate this, in \Cref{example:elimination:square,example:elimination:square:end}, we combine \Cref{alphabetic:thm:criterion} with \Cref{alphabetic:thm:1sol} to give a closed formula for copositivity when $\A_+$ consists of the vertices of a square and $\A_-$ consists of its barycenter.

The paper is organised as follows. In \Cref{sec:copositive:signed:discriminants}, we introduce notation and  the signed $(\A_+,\A_-)$-discriminant, study its Euclidean and Zariski closures, and state and prove   the criterion on  copositivity in \Cref{thm:criterion:copositivity}. \Cref{sec:nonseparable:supports} studies SONC decompositions of nonseparable signed supports: \Cref{subsec:geometry:nonseparable} introduces nonseparable signed supports and some geometric properties, \Cref{subsec:SONC:supports}, formalises the problem of characterising SONC supports, \Cref{subsec:decomposing:critical:system} focuses on SONC decomposition for singular signomials with nonseparable signed supports by adapting  ideas from \cite{Wang_Nonneg}, and, finally, in \Cref{subsec:thm_SONC}, we show that nonseparable signed supports are SONC. 
\Cref{sec:computational} addresses computational aspects of the criterion on copositivity: 
first, we discuss some of its limitations in \Cref{subsec:limitations}, then, we show that nonseparable signed supports behave nicely in a computational sense in \Cref{subsec:criterion:nonseparable:supports}, and finally, in  \Cref{subsec:implementation:SAGE}, we present the Julia package \texttt{CopositivityDiscriminants.jl}, which offers a proof-of-concept implementation of the methods developed in this paper, and we discuss how our criterion compares to other available methods.

\subsection*{Acknowledgments} We thank Carles Checa for discussions on this work, and Simon Telen for   discussions on discriminants and for suggesting \Cref{example:signed:Adiscriminant:not:closed}. 
This work has been supported by the European Union under the
Grant Agreement number 101044561, POSALG\footnote{Views and opinions expressed are those of the
authors only and do not necessarily reflect those of the European Union or European Research
Council (ERC). Neither the European Union nor ERC can be held responsible for them.}.

\section{Copositive signomials and the signed discriminant }\label{sec:copositive:signed:discriminants}

The goal of this section is to provide a characterization of copositivity of Laurent polynomials with real coefficients or more generally, signomials. This is achieved in \Cref{thm:criterion:copositivity} in \Cref{subsec:copositivity:criterion}. Before that, we give general notation agreements in \Cref{subsec:notation}, introduce signed supports in \Cref{subsec:signs}, and discuss   critical systems and the signed discriminant as our main theoretical tools in \Cref{subsec:signed_discriminant}.

\subsection{Notation agreements}\label{subsec:notation}
Given $n\in \N$, we let $[n]=\{1,\dots,n\}$.
The positive and nonnegative orthants are denoted by $\R^n_{>0}$ and $\R^n_{\geq 0}$ respectively, and we denote the complex torus by $(\C^*)^n\coloneqq(\C\setminus\{0\})^n$. 
The Euclidean closure of a set $S$ is denoted  by $\overline{S}$. By $\intt(S)$ and $\relint(S)$ we mean the interior and the relative interior of $S$ respectively and $\lVert v \rVert$ denotes the Euclidean norm of $v\in \R^n$.

For $u,v\in \C^n$, $u\star v$ denotes component-wise multiplication and $\langle u,v\rangle$ the standard scalar product. For $u\in (\C^*)^n$, $u^{-1}=1/u$ is taken component-wise. We let $\mathbbm{1}$ denote the all ones vector of size inferred from the context. The space of real $(n\times m)$-matrices is denoted by $\R^{n\times m}$. It will often be convenient to consider real matrices where rows and columns are indexed by two finite sets $\mathcal{A},\mathcal{B}$, which are not necessarily ordered, and in this case we denote the set by $\R^{\mathcal{A}\times \mathcal{B}}$.  

For $t\in \R_{>0}$ and $h\in \R^n$ a vector, $t^h = ( t^{h_1},\dots,t^{h_n})$. 
For $x\in \R^n_{>0}$   and $a\in \R^{n}$, we refer to $x^a\coloneqq x_1^{a_1}\cdots x_n^{a_n}$ as a monomial, even if the exponent is not an integer vector. For a 
matrix $A\in \R^{n\times k}$  with columns $(a_1,\dots,a_k)$, we denote likewise 
$x^A$  the  monomial map $x^A\colon\R^n_{>0}\to \R^k_{>0}$ with $(x^A)_{j}=x^{a_j}$. 
For a differentiable map $F\colon \R^{n}\to \R^m$, $J_F(x)$ denotes its Jacobian at $x\in \R^n$.

Given a set $\A\subseteq \R^n$, $\conv(\A)$ and $\Aff(\A)$ refer to the convex and affine hulls of $\A$ respectively. For a polytope $P$, ${\rm vert}(P)$ denotes the set of vertices of $P$. Throughout, we consider $P$ to be a face of the polytope $P$ itself.

\subsection{Copositive signomials and sign supports} \label{subsec:signs}
A \term{signomial} is a function $f\colon\R^n_{>0}\to \R$ of the form 
\begin{equation}\label{eq:gen}
f(x)=\sum_{a\in\A_+}c_a\, x^a-\sum_{b\in \A_-}c_b\, x^b
\end{equation}
where $\A_+,\A_-\subseteq\R^n$ are disjoint finite sets, and $c_a>0$, $c_b>0$ for all $a\in\A_+$ and $b\in \A_-$. Observe that the sign of a term of $f$  evaluated at some $x\in \R^n_{>0}$ is the sign of its coefficient. 
We let  $\mathcal{S}_n$ denote the set of all signomials with domain $\R^n_{>0}$, and for a signomial $f\in \mathcal{S}_n$, we let 
\[V_{>0}(f)\coloneqq \{x\in\R^n_{>0}\colon f(x)=0\}\] be the set of positive zeros of $f$.

When the sets $\A_+,\A_-\subseteq\Z^n$ consist of integer tuples, a signomial $f$ is naturally identified with a  Laurent polynomial with real coefficients.

Let us introduce some objects and notation in relation to \cref{eq:gen}.
\begin{enumerate}
\item For a signomial $f$ as in \cref{eq:gen}, the pair $(\A_+,\A_-)$ determined by $f$ is called the \term{signed support} of $f$, whereas the union $\A\coloneqq\A_+\cup\A_-$ is the \term{support} of $f$. We refer to the vector $c=(c_a)_{a\in \A}\in \R^{\A}_{>0}$ as the vector of \term{nonsigned coefficients} of $f$. 

    The dimension of $\A$, or of $(\A_+,\A_-)$, is by definition the dimension of $\conv(\A)$. We say that 
    the (signed) support  is full dimensional if it has dimension  $n$.

   \smallskip
    \item  Given disjoint finite sets $\A_+,\A_-\subseteq \R^n$, the set of signomials with signed support $(\A_+,\A_-)$ is denoted by 
    \[ \mathcal{S}(\A_+,\A_-) \subseteq \mathcal{S}_n \,. \] 
    By identifying  any $f\in \mathcal{S}(\A_+,\A_-)$   with its vector of nonsigned coefficients, we will identify $\mathcal{S}(\A_+,\A_-)$ with $\R^\A_{>0}$, for $\A=\A_+\cup\A_-$, when convenient.

We define the \term{sign vector $\sigma\in\{-1,1\}^{\A}$ associated with  $(\A_+,\A_-)$} by 
\begin{equation}\label{eq:sigma}
\sigma_a=1 \text{ if }a\in\A_+ \quad\text{ and }\quad\sigma_a=-1\text{ if } a\in\A_-\,,  
\end{equation} 
and consider the associated involution
    \begin{align}\label{eq_def_tau}
        \tau\colon\C^\A  \to\C^\A \qquad 
        c  \mapsto c\star \sigma \, ,
    \end{align}
    which maps the orthant of $\R^\A$ containing the coefficients of signomials in $\mathcal{S}(\A_+,\A_-)$ to $\R^{\mathcal{A}}_{>0}$. 

    \item  Given $f\in \mathcal{S}(\A_+,\A_-)$ written as in \cref{eq:gen}, let $\A=\A_+\cup \A_-$ and $\Gamma \subseteq \conv(\A)$ be a nonempty face. The \term{truncation of $f$ at $\Gamma$} is  the signomial
\begin{equation}\label{eq:truncated}
\begin{aligned}
f^\Gamma & \coloneqq\sum_{a\in\A_+\cap \Gamma}c_a\, x^a-\sum_{b\in\A_-\cap\Gamma}c_b\, x^b \ \in \mathcal{S}(\A_+\cap \Gamma,\A_-\cap \Gamma)\,.
\end{aligned}
\end{equation}
The  signed support and support of  $f^\Gamma$  are denoted respectively by 
\[(\A^\Gamma_+,\A^\Gamma_-)\coloneqq(\A_+\cap\Gamma,\A_-\cap\Gamma) \quad \text{and}\quad \A^\Gamma\coloneqq\A\cap\Gamma\, .\]
For every nonempty face $\Gamma$ of $\conv(\A)$, we let 
\[\pi_\Gamma \colon \mathbb{C}^{\A} \to \C^{\A^\Gamma}\] 
be the natural projection map. 
\end{enumerate}

\smallskip
The main object of study of this work are signomials $f\in \mathcal{S}_n$ that are \term{copositive}, 
that is, 
\[ f(x)\ge 0 \qquad \text{for all}\quad x\in \R^n_{>0}\,.\] 

We state now a classical result that explains why considering the sign  of the coefficients of a signomial  is natural  when studying  copositivity of  signomial with fixed support. 

\begin{lemma}\label[lem]{lemma:copositivity:signs:truncations}
    Let $\A_+,\A_-\subseteq \R^n$ be disjoint finite sets, $\A=\A_+\cup\A_-$,  $f\in \mathcal{S}(\A_+,\A_-)$, and $\Gamma\subseteq \conv(\A)$ a nonempty face. For every $x\in \R^n_{>0}$ such that $f^\Gamma(x)\neq 0$, there exists $y\in \R^n_{>0}$ such that 
    \[\sign(f^\Gamma(x))=\sign(f(y))\,.\]
    In particular, if $\vertices(\conv(\A))\cap\A_-\neq \varnothing$, then $f$ is not copositive.
\end{lemma}
\begin{proof}
    See, for instance, \cite[Proposition 2.3]{kinetic_space_multistationarity_dual_phosphorylation}.
\end{proof}

\subsection{Critical systems and discriminants}\label{subsec:signed_discriminant}

 A \term{singular positive zero} of a signomial $f\in \mathcal{S}_n$   is a zero of the  \term{critical system of $f$},
\begin{equation}\label{eq_def_crit_syst}
    \left(f(x),x_1\frac{\partial f}{\partial x_1}(x),\dots, x_n\frac{\partial f}{\partial x_n}(x)\right) \,,
\end{equation}
in $\R_{>0}^n$. 
We denote the set of singular positive zeros of $f$ by $ \Sing_{>0}(f)$.

\begin{remark}\label{rk:copositive_zeros}
If $f$ is a copositive signomial and $f(x)=0$ for $x\in \R^n_{>0}$, then necessarily $x$ is a singular positive zero: a nonzero partial derivative would imply $f$ attains both signs around $x$.
Hence, for a copositive signomial $f$, we have  $x\in\Sing_{>0}(f)$ if and only if $x\in V_{>0}(f)$. 
\end{remark}

The monomials appearing in  each entry of  \eqref{eq_def_crit_syst} are among the monomials of $f$ and are hence determined by the support of $f$. 
More specifically, 
given  disjoint finite sets $\A_+,\A_-\subseteq \R^n$, let $\A=\A_+\cup\A_-$. 

Consider the following matrices:
\begin{itemize}
    \item $A\in \R^{[n]\times \A}$  whose $a$-column is the vector $a$.
    \item $\hat{A}\in\R^{[n+1]\times \A}$ is  obtained by adding a row of ones on top of $A$.
    \item $C\coloneqq\hat{A}\diag(\sigma)\in \R^{[n+1]\times \A}$, where $\sigma$ is the sign vector in \eqref{eq:sigma}. 
\end{itemize}
 The \term{critical system of $(\A_+,\A_-)$} is the system of signomials on $\R^\A_{>0}\times \R^n_{>0}$ defined by
\begin{equation}\label{eq:crit:syst:matrix:form}
F(c,x)\coloneqq C(c\star x^A)\,, \qquad \text{for }(c,x) \in\R^\A_{>0}\times \R^n_{>0}\, .
\end{equation}
The specialization of $F$ to the vector of nonsigned coefficients $c$ of a signomial 
$f\in\mathcal{S}(\A_+,\A_-)$ is precisely 
the critical system of $f$. 

We consider the  incidence correspondence  
\begin{equation}\label{eq_positive_incidence_variety}
    \mathcal{I}_{>0}(\A_+,\A_-)\coloneqq \big\{\, (c,x)\in  \R^\A_{>0} \times \R^n_{>0}: F(c,x)=0\, \big\}\, ,
\end{equation}
 the projection $\pi\colon\C^\A\times\C^n\to\C^\A$ onto the first component, 
and as in \cite[Section 2.3]{Mate2024} we define the \term{signed $(\A_+,\A_-)$-discriminant}  as the set of vectors of nonsigned coefficients of signomials in $\mathcal{S}(\A_+,\A_-)$ that have at least one positive singular zero:
\begin{align}
\label{eq:DegSignedDiscriminant}
\nabla_{>0}(\A_+,\A_-)\coloneqq\pi(\mathcal{I}_{>0}(\A_+,\A_-))\, . 
\end{align}

For  a nonempty face $\Gamma$ of $\conv(A)$, the critical system of  $(\A^\Gamma_+,\A^\Gamma_-)$ is 
\begin{equation}\label{eq:critical_truncated} F^{\Gamma}(c,x) = C^{\Gamma}(c \star x^{A^\Gamma})\qquad \text{for} \quad (c,x)\in \mathbb{R}^{\mathcal{A}\cap\Gamma}_{>0}\times \mathbb{R}^n_{>0}\, ,
\end{equation}
 where 
$C^\Gamma \in \mathbb{R}^{[n+1] \times  \mathcal{A}^ \Gamma}$ and $A^\Gamma \in \mathbb{R}^{[n] \times \mathcal{A}^ \Gamma}$ are the submatrices of $C$ and $A$ formed by the columns indexed by the elements of $\mathcal{A}^\Gamma$. 

We remark that in our definition of the signed $(\A_+,\A_-)$-discriminant, we do not take the Zariski nor the Euclidean closure, as it is common for other types of discriminants such as the $\mathcal{A}$-discriminant and its positive counterpart that we discuss below. 
In what follows, we show that points in the Euclidean closure of the signed $(\A_+,\A_-)$-discriminant correspond to singularities of truncated signomials. We begin with an illustrative example.

\begin{example}\label[ex]{example:signed:Adiscriminant:not:closed}
For every real number $k>1$, consider the signomial
\begin{align*}
    f_k &=(1-\tfrac{1}{k})+(4+\tfrac{1}{k})x_2+x_1^2-(2-\tfrac{1}{k})x_1-4x_1x_2-5x_2^2\\
    &=(x_1+x_2-1)(x_1-5x_2-1+\tfrac{1}{k})
\end{align*}
with signed support $\A_+=\{(0,0),(0,1),(2,0)\}$, $\A_-=\{(1,0),(1,1),(0,2)\}$ and vector of nonsigned coefficients $c_k=(1-\tfrac{1}{k},4+\tfrac{1}{k},1,2-\tfrac{1}{k},4,5)$. 
We have that $c_k\in\nabla_{>0}(\A_+,\A_-)$, as  $f_k$ has a singular zero at $(1-\tfrac{1}{6k},\tfrac{1}{6k})$, where the lines $x_1+x_2-1=0$ and $x_1-5x_2-1+\tfrac{1}{k}=0$ intersect. 

When $k\to \infty$, $c_k$ tends to $c_\infty=(1,4,1,2,4,5)$ and the corresponding signomial $f_\infty =(x_1+x_2-1)(x_1-5x_2-1)$ has no positive singular zero in $\R^n_{>0}$ (the two lines intersect now at $(1,0)$). Therefore $c_\infty\notin \nabla_{>0}(\A_+,\A_-)$, showing that $\nabla_{>0}(\A_+,\A_-)$ is not closed. 
However, for the face $\Gamma = \conv(\{(0,0),(2,0)\})$, the truncation $f_\infty^\Gamma = x_1^2 - 2x_1 + 1$ is singular at $(1,x_2)$ for all $x_2 >0$ from where it follows that $c_\infty\in \nabla_{>0}(\A_+^\Gamma,\A_-^\Gamma)$. 
\end{example}

Using tools from toric geometry, we study now the Euclidean closure of signed discriminants and prove two results that will be used in the proofs of \Cref{lem:SONC:singular:zeros} and \Cref{thm:criterion:copositivity} respectively.

\begin{proposition}\label[prop]{prop:relatively:compact:set}
\label[prop]{prop:closure:Adiscr}
Let $\A_+,\A_-\subseteq \R^n$ be disjoint finite sets and let $\A=\A_+\cup\A_-$.
\begin{enumerate}[label=(\roman*)]
\item  The set 
      \begin{align} \label{eq:DiscUnion}
         \bigcup_{\substack{\varnothing\neq \Gamma \text{ face of}\\   \conv(\A)}} \pi_\Gamma^{-1}\Big(\nabla_{>0}(\A^\Gamma_+,\A^\Gamma_-)\Big)
     \end{align}
     is closed in $\R^{\A}_{>0}$ and contains the closure of $\nabla_{>0}(\A_+,\A_-)$ in the Euclidean topology of $\R^{\A}_{>0}$.

    \item If $\A$ is full dimensional and  $\A_-\subseteq\intt(\conv(\A_+))$, then  $V_{>0}(f)$ is compact in $\R^n_{>0}$ for all $f \in \mathcal{S}(\A_+,\A_-)$.

    \item 
    Let $\{\gamma_k\}_{k\in \mathbb{N}}$ be a sequence with limit $c\in \R^\A_{\geq 0}$  that satisfies $c_a\neq 0$ for all $a\in \vertices(\conv(\A))$. 
    If for all $k\in \N$, the signomial $f_k\in \mathcal{S}(\A_+,\A_-)$ with vector of nonsigned coefficients $\gamma_k$ satisfies $V_{>0}(f_k)\neq \varnothing$, then $c_b\neq 0$ for some $b\in \A_-$. 
  
\end{enumerate}
\end{proposition}
\begin{proof}
We start by showing a general fact that will be used to prove the three statements.
For a matrix $B\in \R^{[s]\times \A}$, consider  the systems of signomials of the form $G=B(c\star x^A)$ with $c\in \R^\A_{\geq 0}$. 
Let $Y_{>0}$  be the image of the monomial map 
\[\varphi \colon \R^{n+1}_{>0} \to \R^\A_{>0}\,  \qquad \hat{x} = (u,x)  \mapsto \hat{x}^{\hat{A}} = (u,\dots,u)\star x^A\, \]
and let $Y_{\geq 0}$ be its Euclidean closure  in $\R^\A$. 
The set $Y_{\geq 0}$ is called an \emph{irrational affine toric variety} in \cite{SottilePir},  while in the case where $A$ has integer entries, it
is termed the \emph{nonnegative affine toric variety} in \cite{telen2023positive}.
By \cite[Equation (4)]{SottilePir} (see also  \cite[Equation (9)]{SottilePostinghelVillamizar}), for every nonzero $y \in Y_{\geq 0}$, there exists a nonempty face $\Gamma$ of $\conv(\mathcal{A})$ and $\hat{x} = (u,x) \in \mathbb{R}^{n+1}_{>0}$ such that 
\begin{align}
\label{eq:toricproof}
    y_{a} = u\,  x^{a} \quad  \text{if} \quad a \in \Gamma \cap \mathcal{A} \qquad \text{and} \qquad y_{a} = 0 \quad \text{otherwise\,.} 
\end{align}

Let $\Delta = \{ \, y \in \R^\A_{\geq 0} :  \sum_{a\in \A} y_a = 1 \, \} \subseteq \mathbb{R}^\A$  be the probability simplex.
We will use that the following fact holds. 

($\ast$) \emph{Given a   sequence  $\{\gamma_k\}_{k\in \mathbb{N}}$ converging to some $c\in \R^\A_{\geq 0}$ and a sequence  $\{y_k\}_{k\in \mathbb{N}} \subseteq   Y_{>0}\cap \Delta$ such that $B(\gamma_k\star y_k)=0$ for all $k$, there exists a subsequence of $\{y_k\}$ converging to some  $y \in Y_{\geq 0} \cap \Delta$ that satisfies $B(c\star y)=0$.} 

To see why this holds, we first consider a  subsequence  $\{y_{i_k}\}$  of $\{y_k\}_{k\in \mathbb{N}}$ that converges to some $y\in Y_{\geq 0} \cap \Delta$, which exists as this set is compact. 
For any $i_k$, using that $0=B(\gamma_{i_k}\star y_{i_k}) = B\diag(y_{i_k})\gamma_{i_k}$, we have
     \begin{align*}
          0 \leq  \lVert B(c\star y) \rVert  &= \lVert B\diag(y)c \rVert = \lVert B\diag(y)c - B\diag(y_{i_k})c + B\diag(y_{i_k})c -  B\diag(y_{i_k})\gamma_{i_k}  \rVert  \\
           &\leq \lVert B\diag(y-y_{i_k}) \rVert \, \lVert c  \rVert \,  + \, \lVert B\diag(y_{i_k}) \rVert \, \lVert c -  \gamma_{i_k}  \rVert \, .       \end{align*}
As both summands converge to zero,
 it must hold that $B(c\star y)=0$. 

 \smallskip
 We are now ready to prove the statements. 
For part (i), consider the matrix $C$ defining the critical system $F(c,x)= C(c\star x^A)$ of $(\A_+,\A_-)$ as in \eqref{eq:crit:syst:matrix:form}. 
 We observe that
 \begin{equation}\label{eq:nabla_Y}
    \begin{aligned}
    \nabla_{>0}(\A_+,\A_-) &= \{ \, c \in \mathbb{R}^{\A}_{>0} : C (c\star x^A) = 0  \text{ for some }x \in \mathbb{R}^n_{>0}\}  \\ 
    &= \{ \, c \in \mathbb{R}^{\A}_{>0}  :  u\,  C(c\star x^A) = 0 \text{ for some }(u,x) \in \mathbb{R}^{n+1}_{>0}\} \\
     &= \{ \, c \in \mathbb{R}^{\A}_{>0} :  C(c\star \hat{x}^{\hat{A}}) = 0 \text{ for some } \hat{x} = (u,x) \in \mathbb{R}^{n+1}_{>0}\} \\ 
       &= \{ \, c \in \mathbb{R}^{\A}_{>0} :  C(c\star y) = 0  \text{ for some } y \in Y_{>0} \cap \Delta    \}\,,
       \end{aligned}
       \end{equation}
     where in the last equality we use  that
      $y \in Y_{>0}$ if and only if $\tfrac{y}{\sum_{a\in \A} y_a}  \in Y_{>0}$ as $Y_{>0}$ is a cone.

Let $\{\gamma_k\}_{k\in \mathbb{N}} \subseteq \nabla_{>0}(\A_+,\A_-)$ be a convergent sequence with limit $c \in \mathbb{R}^{\A}_{>0}$. By \eqref{eq:nabla_Y}, for every $k$ there exists $y_k \in Y_{>0}\cap \Delta$ such that $C(\gamma_k\star y_k) = 0$. 
We are in the general situation above with $B=C$ and hence, using the fact ($\ast$),  there exists 
$y\in Y_{\geq 0} \cap \Delta$ such that $C(c\star y)=0$. 
For this $y$, 
there exists $\hat{x} = (u,x) \in \R^{n+1}_{>0}$ and a nonempty face $\Gamma$ of $\conv(\A)$ satisfying \eqref{eq:toricproof}, and hence
      \begin{equation}\label{eq:toricproof2} 
      0=C(c\star y) =   C^\Gamma (\pi_\Gamma(c)\star \hat{x}^{\hat{A}^\Gamma})\, .
      \end{equation} 
We conclude that  $\pi_\Gamma(c) \in \nabla_{>0}(\A^\Gamma_+,\A_-^\Gamma)$, that is, $c\in \pi_\Gamma^{-1}(\nabla_{>0}(\A^\Gamma_+,\A_-^\Gamma))$. 

Applying this argument to the truncated signed support of a face $\Gamma$ as well, 
we obtain that  for every $c \in \pi_\Gamma^{-1}\big(\,\overline{\nabla_{>0}(\A^\Gamma_+,\A^\Gamma_-)}\,\big)$, there exists a nonempty face $\Gamma' \subseteq \Gamma$ such that $c \in \pi_{\Gamma'}^{-1}\big(\,\nabla_{>0}(\A^{\Gamma'}_+,\A^{\Gamma'}_-)\,\big)$. 
Therefore, 
\[  \bigcup_{\substack{\varnothing\neq \Gamma \text{ face of}\\   \conv(\A)}} \pi_\Gamma^{-1}\big(\nabla_{>0}(\A^\Gamma_+,\A^\Gamma_-)\big) = 
\bigcup_{\substack{\varnothing\neq \Gamma \text{ face of}\\   \conv(\A)}} \pi_\Gamma^{-1}\big(\,\overline{\nabla_{>0}(\A^\Gamma_+,\A^\Gamma_-)}\,\big)\, \supseteq \overline{\nabla_{>0}(\A_+,\A_-)}\,.  \]
We conclude that  \cref{eq:DiscUnion}  is a finite union of closed sets, hence is closed, showing part (i).

\smallskip
To show parts (ii) and (iii), we start with some common ingredients. For (ii), we consider $f\in \mathcal{S}(\A_+,\A_-)$ with  vector of nonsigned coefficients $c$, and for (iii) we consider the signomial $f$ with vector of coefficients $c\star \sigma$, with $\sigma$  as in \eqref{eq:sigma}. In both cases,  $f(x)=\langle c \star \sigma  , x^A \rangle$, so we have a signomial system where the coefficient matrix $B$ above has one row equal to $\mathbbm{1}\star \sigma$. 

Given $x\in \R^n_{>0}$, we let $u_x  \coloneqq \tfrac{1}{\sum_{a\in \mathcal{A}} x^a}> 0$ and consider the map
\[ \phi \colon \R^n_{>0} \rightarrow Y_{>0} \cap \Delta \qquad x \mapsto \varphi(u_x,x)\,. \]
As  the map   sending $x$ to $(u_x,x)$ is a homeomorphism onto its image, and $\varphi$ is a homeomorphism between $\R^{n+1}_{>0}$ and $Y_{>0}$ since $\mathcal{A}$ is full dimensional,  $\phi$ is also a homeomorphism, and hence $V_{>0}(f)$ and $\phi(V_{>0}(f))$ are homeomorphic. 
Observe that 
\begin{equation}\label{eq:phi}
y\in \phi(V_{>0}(f)) \quad \Leftrightarrow \quad y\in Y_{>0} \cap \Delta \text{ and }\langle c \star \sigma  , y \rangle=0\, . 
\end{equation}

Part (ii) follows if we show that $\phi(V_{>0}(f))$
is compact in $Y_{>0} \cap \Delta$. To this end, consider a sequence $\{y_k\}_{k\in \N} \subseteq \phi(V_{>0}(f))$ and the convergent constant sequence $\gamma_k=c$.  Using \eqref{eq:phi} for each $y_k$, the   fact ($\ast$) 
gives the existence of a subsequence of $\{y_k\}$ that converges to some $y\in Y_{\geq 0} \cap \Delta$ that satisfies 
$\langle c \star \sigma  , y \rangle=0$. 
Let $\Gamma$ be the nonempty face of $\conv(\A)$ satisfying \eqref{eq:toricproof} for this $y$,  and thus $\langle \pi_\Gamma(c \star \sigma)  , \pi_\Gamma(y) \rangle=0$.
As $\pi_\Gamma(y)$ is a vector with positive entries, $\pi_\Gamma(c \star \sigma)$ must have entries of both signs, and by the assumption that the boundary of $\conv(\A)$ contains no point of $\A_-$, this implies that $\Gamma=\conv(\A)$ again by \eqref{eq:toricproof}. Hence, $y \in Y_{>0} \cap \Delta$. Using \eqref{eq:phi} it follows that $y\in \phi(V_{>0}(f))$.
As every sequence in $\phi(V_{>0}(f))$ has a convergent subsequence,  the set is compact. This concludes the proof of part (ii).

Part (iii) is shown similarly. For every $k$, let $x_k\in V_{>0}(f_k)$, $y_k\coloneqq \phi(x_k) \in \phi(V_{>0}(f))$, and consider the  sequence $\{\gamma_k\}$. Using \eqref{eq:phi} and ($\ast$) ,  there exists $y\in Y_{\geq 0} \cap \Delta$ such that
$\langle c \star \sigma  , y \rangle=0$, and a nonempty face $\Gamma$ of $\conv(\A)$ such that $\langle \pi_\Gamma(c \star \sigma)  , \pi_\Gamma(y) \rangle=0$, and $\pi_\Gamma(y)\neq 0$ has positive entries.  As $c_a\neq 0$ if $a\in \vertices(\conv(\A))$ and   $\vertices(\Gamma)\subseteq \vertices(\conv(\A))$, we have that $\pi_\Gamma(c \star \sigma)\neq 0$ and hence it must have positive and negative entries. In particular, $c_b\neq 0$ for some $b\in \A_-$.
\end{proof}

If $\Gamma$ is a vertex of $\conv(\A)$, then $\nabla_{>0}(\A^\Gamma_+,\A^\Gamma_-)=\varnothing$. Hence it follows from \Cref{prop:relatively:compact:set} that any sequence in $\nabla_{>0}(\A_+,\A_-)$ will converge to some point in $\nabla_{>0}(\A^\Gamma_+,\A^\Gamma_-)$ for some face $\Gamma$ of dimension at least one. In particular,  the signed discriminant is closed in $\R^\A_{>0}$ if $\A$ has dimension one.

The signed $(\A_+,\A_-)$-discriminant is closely related to the positive $\A$-discriminant from~\cite{bihan2022new}, which is defined as the Euclidean closure in $(\R\setminus \{0\})^\A$ of the coefficient vectors of all signomials
with support $\A$ that have at least one positive singular zero. The crucial property of the positive $\A$-discriminant is that it detects the changes in the topology of the zero set of a signomial as the coefficients vary, as we explain in the following proposition.

\begin{proposition}[{\cite[Proposition 2.10]{bihan2022new}}]
\label[prop]{prop:Bihan} Let $\A_+,\A_-\subseteq \R^n$ be disjoint finite sets and let $\A=\A_+\cup\A_-$. If the nonsigned coefficients of $f,f'\in\mathcal{S}(\A_+,\A_-)$ lie in the same connected component of \[
\R^\A_{>0}\setminus \Big( \bigcup_{\substack{\varnothing\neq \Gamma \text{ face of }\\  \conv(\A) }} \pi_\Gamma^{-1}\big(\nabla_{>0}(\A^\Gamma_+,\A^\Gamma_-)\big) \Big)\,,
\]
then $V_{>0}(f)$ and $V_{>0}(f')$ are homeomorphic.
\end{proposition}
\begin{proof}
As by \Cref{prop:closure:Adiscr}, the set in \eqref{eq:DiscUnion} is closed,  the statement follows from restricting  \cite[Proposition 2.10]{bihan2022new}   to the orthant of $\R^{\A}$ determined by the map  $\sigma$ in \eqref{eq:sigma} and using the associated involution $\tau$ in \cref{eq_def_tau} that sends this orthant to $\R_{>0}^\A$. 
\end{proof}

 The signed $(\A_+,\A_-)$-discriminant is also related to the classical   $\A$-discriminant from \cite{gelfand}  for a finite set $\A\subseteq\Z^n$. 
 In this case, we consider   Laurent polynomials $f$ with support $\A$, and let 
$\Sing(f)$ be the set of zeros of the critical system \eqref{eq:crit:system:f} of $f$ in $(\C^*)^n$. 
Then the $\A$-discriminant is the set
\[
\nabla_\C(\A)\coloneqq\Zar\left(\left\{c\in\C^\A\:\colon\:\Sing\left(\sum_{a\in \A}c_a\, x^a\right)\neq \varnothing\right\}\right)\,,
\]
where  Zar refers to the Zariski closure in $\C^A$. 
Using this, we determine the Zariski closure of $\nabla_{>0}(\A_+,\A_-)$ for full dimensional signed supports with integer exponents.

\begin{proposition}\label[prop]{prop:zariski:closure:Adiscr}
Let $\A_+,\A_-\subseteq \Z^n$ be disjoint finite sets such that $\A=\A_+\cup\A_-$ is full dimensional. Let $\tau$ be the associated involution from \cref{eq_def_tau}. If $\nabla_{>0}(\A_+,\A_-)\neq \varnothing$, then 
\[ \tau(\nabla_{\C}(\A)) = \Zar(\nabla_{>0}(\A_+,\A_-))\,, \]
where  the Zariski closure is taken in $\C^\A$.
\end{proposition}
\begin{proof} 
Let $F$ be defined as in \cref{eq:crit:syst:matrix:form} and extended over $\C^\A\times(\C^*)^n$, and consider the incidence variety $\mathcal{I}(\A)\coloneqq\{(c,x)\in \C^\A\times(\C^*)^n\:\colon\:F(c,x)=0\}$ and the projection $\pi\colon\C^\A\times\C^n\to\C^\A$. By definition of $\nabla_\C(\A)$, we have \[\nabla_\C(\A)=\Zar(\tau(\pi(\mathcal{I}(\A))))=\tau(\Zar(\pi(\mathcal{I}(\A)))).\] 
As $\conv(\A)$ is full dimensional, the matrix $C$ from \cref{eq:crit:syst:matrix:form} has full row rank. 
As $F$ is a vertically parametrized system in the sense of \cite{feliu_henriksson_pascualescudero_generic_consistency}, it follows   from  \cite[Theorem 3.1]{feliu_henriksson_pascualescudero_generic_consistency} that $\mathcal{I}(\A)$ is irreducible and nonsingular. Since $\nabla_{>0}(\A_+,\A_-)\neq \varnothing$, we must have $\mathcal{I}_{>0}(\A_+,\A_-)\neq \varnothing$, and hence it follows from \cite[Remark 2.9]{feliu_henriksson_pascualescudero_generic_consistency} that $\mathcal{I}_{>0}(\A_+,\A_-)$ is Zariski dense in $\mathcal{I}(\A)$. With this, we obtain  
\begin{align*}
        \Zar(\nabla_{>0}(\A_+,\A_-)) & =\Zar(\pi(\mathcal{I}_{>0}(\A_+,\A_-)))=\Zar(\pi(\Zar(\mathcal{I}_{>0}(\A_+,\A_-))))\\
        &=\Zar(\pi(\mathcal{I}(\A)))=\tau(\nabla_\C(\A))\,,
    \end{align*}  
    where in the second equality we use that $\pi$ is a continuous map.
\end{proof}

We conclude this subsection with  a useful lemma and remark that  will be used repeatedly throughout this work.

\begin{lemma}\label[lem]{lemma:affine:transformation:Adiscriminant}
 Let $\A_+,\A_-\subseteq \R^n$ be disjoint finite sets, $\A=\A_+ \cup \A_-$, 
 and $\psi$ be an affine transformation on $\R^n$. For  $f=\sum_{a\in\A_+}c_a\, x^a-\sum_{b\in\A_-}c_bx\, ^b\in \mathcal{S}(\A_+,\A_-)$, consider the signomial
    \[
    f^{(\psi)}\coloneqq\sum_{a\in \A_+}c_a\, x^{\psi(a)}-\sum_{b\in \A_-}c_b\, x^{\psi(b)} \quad \in\   \mathcal{S}(\psi(\A_+),\psi(\A_-)) \, . 
    \]
    Then, the following statements hold:
    \begin{itemize}
          \item $\Sing_{>0}(f)\neq\varnothing$ if and only if $\Sing_{>0}(f^{(\psi)})\neq\varnothing$. 
        \item $f$ is copositive if and only if $f^{(\psi)}$ is.
\item With the natural identification of $\R^{\A}_{>0}$ with $\R^{\psi(\A)}_{>0}$ it holds that
    \[ \mathcal{S}(\A_+,\A_-)=\mathcal{S}(\psi(\A_+),\psi(\A_-)) \,\qquad\text{and}\qquad  \nabla_{>0}(\A_+,\A_-)=\nabla_{>0}(\psi(\A_+),\psi(\A_-)) \, . \]
    \end{itemize}
\end{lemma}
\begin{proof}
Let  $M\in\R^{n\times n}$ be an invertible matrix and $v\in \R^n$ be such that $\psi(a) = Ma + v$ for every $a \in \mathbb{R}^n$.
The statements follow from the equality  $f^{(\psi)}(x)=x^{v}  f(x^{M})$, and from the fact that for $x\in \R^n_{>0}$ such that $f^{(\psi)}(x)=0$, it holds that $J_{f^{(\psi)}}(x) = x^v J_f(x^M) \diag(x^M) M^\top \diag(x^{-1})$. See also \cite[Proposition 2.3]{Mate2024}. 
\end{proof}

\begin{remark}\label[rmk]{rk:affine:transformation:lower:dimension}
By \Cref{lemma:affine:transformation:Adiscriminant}, when studying $\nabla_{>0}(\A_+,\A_-)$, we can assume without loss of generality that $\A=\A_+\cup \A_-$ is full dimensional. Indeed, if $\A$ has dimension $d<n$, there exists an affine transformation such that $\psi(\A)\subseteq\{x\in\R^n\:\colon\:x_{d+1}=\cdots=x_n=0\}\cong\R^d$. The signed discriminant is invariant under this transformation. 
Furthermore $f^{(\psi)}$ depends only on $x_{1},\dots,x_d$ and hence can be regarded as a signomial in $\R^d$, which has  full dimensional support. 
\end{remark}

\subsection{A criterion for copositivity}\label{subsec:copositivity:criterion}
In this subsection we state and prove one of the main results of this work: a criterion to determine whether a signomial is copositive. 

Let $\A_+,\A_- \subseteq \R^n$ be disjoint finite sets and let    $\A=\A_+\cup \A_-$. A map   $h\colon\A\to \R_{\geq 0}$ is called a \term{height function lifting $\A_-$} if
\begin{equation}\label{eq:h}
h(a)>0 \quad\text{if }a\in \A_- \qquad \text{and}\qquad h(a)=0 \quad\text{if }a\notin \A_-\, . 
\end{equation}
For any such $h$ and $f\in\mathcal{S}(\A_+,\A_-)$ with vector of nonsigned coefficients $c$,  we consider the path $\{ c\star t^h\:\colon\: t\in \R_{>0}\}$ in $\R^{\A}_{>0}$. The signomials with vector of nonsigned coefficients $c\star t^h$, namely,
\begin{equation}\label{eq:Viro}
f_{t}(x) \coloneqq\sum_{a\in \A_+}c_a\,x^a-\sum_{b\in\A_-}c_b\, t^{h(b)}x^b \, , 
\end{equation}
define  a path in $\mathcal{S}(\A_+,\A_-)$. Perturbations of signomials as in \eqref{eq:Viro} have been commonly used in real algebraic geometry, where  they for example play a crucial role in Viro's patchworking argument \cite{ViroThesis}.

The following theorem characterizes copositivity for all signomials whose signed support satisfies that $\vertices(\conv(\A))\subseteq \A_+$; equivalently, $\A_-\subseteq\conv(\A_+)$. Note that any signomial not satisfying this assumption cannot be copositive by \Cref{lemma:copositivity:signs:truncations}.

\begin{theorem}\label[thm]{thm:criterion:copositivity}
Let $\A_+,\A_- \subseteq \R^n$ be disjoint finite sets and let $\A=\A_+\cup \A_-$. Assume that $\A_-\neq \varnothing$ and $\A_-\subseteq\conv(\A_+)$. 
Let
\[
\mathcal{F}\coloneqq\{\,\Gamma\text{ face of }\conv(\A): \Gamma\cap\A_-\neq\varnothing \text{ and }\Gamma \text{ is the smallest face containing }\Gamma\cap\A_-\, \}\:.
\]
 Let $f\in\mathcal{S}(\A_+,\A_-)$  have vector of nonsigned coefficients $c\in \R^{\A}_{>0}$,  
 and  let $h\colon \A\to \R_{\geq 0}$ be a height function lifting $\A_-$. Consider the set 
\[
T\coloneqq\bigcup_{\Gamma\in \mathcal{F}}\Big\{t\in \R_{>0} \, :  \pi_{\Gamma}(c\star t^h) \in \nabla_{>0}(\A_+^\Gamma,A_-^\Gamma)\Big\}\,.
 \] 
 Then the following statements hold:
 \begin{enumerate}[label=(\roman*)]
     \item $T$ has a minimum   $t_*$, which also satisfies that $[t_*,\infty)=\{t\in \R_{> 0} : V_{>0}(f_t)\neq \varnothing\}$.
     \item $f$ is copositive if and only if $t_*\geq 1$.
 \end{enumerate}
 \end{theorem}
 \begin{proof}
Observe that the set in \eqref{eq:DiscUnion} satisfies
\begin{equation}\label{eq:U}
D\coloneqq \bigcup_{\substack{\varnothing\neq \Gamma \text{ face of}\\   \conv(\A)}} \pi_\Gamma^{-1}\big(\nabla_{>0}(\A^\Gamma_+,\A^\Gamma_-)\big)
= \bigcup_{\substack{\Gamma \in \mathcal{F}}} \pi_\Gamma^{-1}\big(\nabla_{>0}(\A^\Gamma_+,\A^\Gamma_-)\big)\,, 
\end{equation}
as $\nabla_{>0}(\A^\Gamma_+,\A^\Gamma_-)=\varnothing$ if $\Gamma \notin \mathcal{F}$. This is immediate if $\Gamma\cap \A_-=\varnothing$, and if $\varnothing\neq \Gamma\cap \A_-$ is contained in a proper face of $\Gamma$, this follows as 
$\Gamma\cap \A$ 
satisfies \cite[Theorem 3.3(iii)]{Mate2024}.

We conclude that $t\in T$ if and only if $c\star t^{h}\in D$. 
By letting $\rho_c\coloneqq \{ c\star t^{h} :  t\in \R_{>0}\}$, we have in particular  that $T$ is the preimage of  $D \cap  \rho_c$ by the continuous map $\R_{>0} \rightarrow  \R_{>0}^{\A}$ sending $t$ to $c\star t^h$. 
As $D$ is closed by \cref{prop:closure:Adiscr}(i) and $\rho_c$ is  closed in $\R^\A_{>0}$, so is $T$. 
To  show that $T$ has a minimum, it is thus enough to show that $T$ is bounded from below and nonempty.  

First, we show that $V_{>0}(f_t)\neq \varnothing$ for all $t>0$ large enough. On one hand, since all vertices of $\conv(\A)$ belong to $\A_+$, \Cref{lemma:copositivity:signs:truncations}
tells us that for all  $t>0$, there exists $y\in \R^n_{>0}$ such that $f_t(y)>0$. 
On the other hand, we have that
\[f_t(\mathbbm{1})=\sum_{a\in\A_+}c_a-\sum_{b\in \A_{-}}c_b\, t^{h(b)}\]
 tends to $-\infty$ for $t\to \infty$. So, there exists $t_\infty\in \R_{>0}$ such that for all $t>t_\infty$ it holds that $f_{t}(\mathbbm{1})<0$ and $f_{t}(y)>0$ for some $y$. By continuity,   $V_{>0}(f_{t})\neq \varnothing$ for all $t>t_\infty$.

Consider
\[ N\coloneqq \{t\in \R_{> 0} : V_{>0}(f_t)\neq \varnothing\}\,,\qquad   t_0 \coloneqq {\rm inf}\,N \geq 0  \,, \]
 and note that $t_0$ exists as we have already seen that $N\neq \varnothing$. 
The set $N$ is an interval of the form 
$[t_0,\infty)$ or 
$(t_0,\infty)$. Indeed, if  $t>t_0$, then let $t'$ be such that $t_0<t'<t$ and $t'\in N$, and let $x\in \R^n_{>0}$ satisfy $f_{t'}(x)=0$. Using that $f_{t}(x)<f_{t'}(x)=0$, 
we conclude that $f_t$ attains negative values. Since $f_t$ also attains positive values, $V_{>0}(f_t)\neq 0$ and hence $t\in N$. 

We show now that $t_0>0$. 
Consider a   decreasing sequence $\{t_k\}_{k\in \N}$ converging to $t_0$ and the sequence $\gamma_k\coloneqq c\star t^h_k$ converging to $c\star t_0^h\in \R^\A_{\geq 0}$ (where  $0^0=1$ if $t_0=0$). 
As $V_{>0}(f_{t_k})\neq \varnothing$ and 
$(\gamma_k)_a\neq 0$ for all $a\in \in \vertices(\conv(A))\subseteq \A_+$,    \Cref{prop:relatively:compact:set}(iii) gives that $(c\star t_0^h)_b\neq 0$ for some $b\in \A_-$. As $h(b)>0$, it follows that  $t_0> 0$, 
and hence $V_{>0}(f_t)=\varnothing$ for all $0<t< t_0$. This gives in particular that $T$ is bounded from below by $t_0$. 

That   $T\neq \varnothing$   follows now from  \Cref{prop:Bihan}, as the topology of $V_{>0}(f_t)$
changes along the path $\rho_c$: it differs for $t<t_0$ and $t>t_\infty$. This concludes the proof of the first part (i). 

\smallskip
For the second part of (i),  clearly $t_0\leq t_*$. 
By definition of $t_*$, $c\star t^h \notin D$ for all $0<t<t_*$  and hence, 
 by \Cref{prop:Bihan},  $V_{>0}(f_t)=\varnothing$ for all $t<t_*$, as this is the case for $t$ small enough. This gives that $t_*\leq t_0$. Putting the two inequalities together we obtain that $t_*= t_0$ and as $t_*\in \{t\in \R_{>0} : V_{>0}(f_t)\neq \varnothing\}$, part (i) follows.

\smallskip
  Part (ii) is a consequence of the   equivalences
  \begin{align}
      \text{$f$ is copositive}\quad & 
      \Leftrightarrow \quad \text{$f_{t}$ is positive on $\R^n_{>0}$ for all $t<1$} \label{eq:equiv1} \\
      & \Leftrightarrow \quad \text{$V_{>0}(f_{t})=\varnothing$ for all   $t<1$} \label{eq:equiv2} \\
      & \Leftrightarrow \quad t_*\geq 1\,   \label{eq:equiv4}
  \end{align}
  derived as follows. For \eqref{eq:equiv1}, the choice of $h$ gives that $f_t(x)>f(x)$  for all $t<1$ and $x\in \R^n_{>0}$ and the equivalence follows by continuity of $f_t$ with respect to $t$. 
In \eqref{eq:equiv2}, the forward implication is clear. For the reverse implication, if $f_t$ is not positive for some $t<1$, then  $V_{>0}(f_t)\neq \varnothing$ as   $f_t$  attains positive values. 
Finally, \eqref{eq:equiv4} holds as $t_*=t_0$. 
 \end{proof}

We will discuss in \Cref{sec:computational} how to determine $t_*$ from \Cref{thm:criterion:copositivity} in practice, but note that finding $t_*$   amounts to solving finitely many critical systems in the variables $(t,x)$, one for each   $\Gamma$  in $\mathcal{F}$. 
 When $\A_-\subseteq \relint(\conv(\A_+))$, we have that  $\nabla_{>0}(\A_+^\Gamma,A_-^\Gamma)\neq \varnothing$ only for $\Gamma=\conv(\A)$, and the criterion in  \Cref{thm:criterion:copositivity} gets simplified in the sense that $t_*$ can be found by solving only one system. This is covered in the next corollary.

\begin{corollary}\label[cor]{cor:copositivity:interior}
    Let $\A_+,\A_- \subseteq \R^n$ be disjoint finite sets and let $\A=\A_+\cup \A_-$. Assume that $\A_-\neq \varnothing$ and $\A_-\subseteq\relint(\conv(\A_+))$. 
    Let $f\in\mathcal{S}(\A_+,\A_-)$  have vector of nonsigned coefficients $c\in \R^{\A}_{>0}$ and 
    let $h\colon \A\to \Z$ be a height function lifting $\A_-$. Consider the set 
\[
T\coloneqq\big\{t\in \R_{>0} :  c\star t^h \in \nabla_{>0}(\A_+,\A_-) \big\} \,.
 \] 
 Then $T$ has a minimum $t_*$, and $f$ is copositive if and only if $t_*\geq 1$.
\end{corollary}

The copositivity criterion from \Cref{thm:criterion:copositivity} and \Cref{cor:copositivity:interior} can be combined with elimination techniques to find  conditions for copositivity on the nonsigned coefficients of polynomials with given signed support.. 
This is illustrated in the next examples.   First, we recall that for points $a_1,\dots,a_{n+1}\in \R^n$ whose convex hull is a full dimensional simplex and $b\in \conv(\{a_1,\dots,a_{n+1}\})$, the \term{barycentric coordinates} of $b$ with respect to  $\{a_1,\dots,a_{n+1}\}$ are the unique 
$\lambda^b_{a_1},\dots,\lambda^b_{a_{n+1}}\in [0,1]$ such that
\[ b = \sum_{i=1}^{n+1} \lambda^b_{a_i} a_i\,, \qquad \ \sum_{i=1}^{n+1} \lambda^b_{a_i} =1\,  . \]

\begin{remark}\label[rmk]{rk:barycentric}
Given a vector $v\in\R^n$, let $\hat{v}\in\R^{n+1}$ be the vector obtained by adding  a $1$ at the top of $v$.
Given affinely independent points $a_1,\dots,a_{n+1}\in \R^n$, there exists an invertible square matrix $M$ of size $n+1$ sending $\hat{a}_i$ to the $i$-th canonical vector of $\R^{n+1}$. 
For any  $b\in \conv(\{a_1,\dots,a_{n+1}\})$, the barycentric coordinates of $b$ satisfy 
$\hat{b}=  \sum_{i=1}^{n+1} \lambda^b_{a_i} \hat{a}_i$, and hence, $M\hat{b}=(\lambda^b_{a_1},\dots,\lambda^b_{a_{n+1}})^\top$. In particular, this gives that  the row reduced  echelon form of the matrix $C$ in \cref{eq:crit:syst:matrix:form} of the critical system when $\A_+ = \{a_1,\dots,a_{n+1}\}$ and $\A_-=\{b\}$   is
    \[
   \begin{pmatrix}[c|c]
        {\rm Id}_{n+1} &   -\gamma
    \end{pmatrix}\in \R^{(n+1)\times (n+2)}\, \qquad \gamma =(\lambda^{b}_{a_1},\dots,\lambda^{b}_{a_{n+1}})^\top.  
    \]

Additionally, as $\{a_1,\dots,a_{n+1}\}$ are affinely independent, the matrix $N\in \R^{n\times n}$ whose $i$-th column is given by $[a_{i+1}-a_1]$ is invertible. Then, the affine transformation $\psi(a)\coloneqq N^{-1}(a-a_1)$ sends $\psi(a_1)=0$ and $\psi(a_i)$ to the $(i-1$)-th canonical vector of $\R^n$. For $b\in \conv(\{a_1,\dots,a_n\})$, we obtain
\[
\psi(b)=N^{-1}(b-a_1)=N^{-1}\left(\sum_{i=1}^{n+1}\lambda_{a_i}^b(a_i-a_1)\right)=\sum_{i=1}^{n+1}\lambda_{a_i}^bN^{-1}(a_i-a_1)=(\lambda_{a_2}^b,\dots,\lambda_{a_{n+1}}^b)\,.
\]
\end{remark}

We discuss now two examples to illustrate how \Cref{cor:copositivity:interior} can be used to obtain closed formulas on the nonsigned coefficients that characterize if a signomial is copositive. First, we consider the simplest signed support and rederive the well-known circuit number. Second, we show that signomials whose support is a square with an interior point at its barycenter also admit a nice closed formula for copositivity in terms of circuit numbers. 

\begin{example}[The circuit number] \label[ex]{example:circuit:number}
We consider $n$-variate signomials with full dimensional signed support $(\A_+,\A_-)$ such that $\A_+$ is the set of vertices of a simplex and $\A_-$ consists of one point in the interior of $\conv(\A_+)$. Whether such a signomial is copositive is determined by the circuit number \cite{craciun_pantea_koeppl,wolff_nonneg}, which we rederive now using \Cref{cor:copositivity:interior}. 

By \Cref{lemma:affine:transformation:Adiscriminant}, copositivity is invariant under affine transformations of the signed support. Therefore, using the affine transformation $\psi$ from \Cref{rk:barycentric}, we can assume that
 \[
 f=c_1+\sum_{i=2}^{n+1}c_{i}x_{i-1}-d\, x^\gamma \qquad \gamma=(\lambda^{b}_{a_2},\dots,\lambda^{b}_{a_{n+1}})\,.
 \]
 
 We know from \Cref{cor:copositivity:interior} that the critical system
 \begin{align*}
 \big(c_1+c_2x_1+\dots+c_{n+1}x_n-d\,t\, x^\gamma\, , \, 
 c_2x_1-\lambda^b_{a_2} d\,t\,x^\gamma \, , \,  \dots \, , \,
 c_{n+1}x_n-\lambda^b_{a_{n+1}}d\,t\, x^\gamma \big) 
\end{align*}
 has at least one positive zero in the variables $(t,x_1,\dots,x_n)$ for all choices of nonsigned coefficients. 
Subtracting the last $n$ equations from the first one, we get the equality $d\,t=\tfrac{c_1}{\lambda^b_{a_1}}x^{-\gamma}$. This, together with the $(i+1)$-th equation and the properties of the barycentric coordinates, give that
\[
x_i=\frac{c_1\lambda^b_{a_{i+1}}}{\lambda^b_{a_1}} \quad \text{ for }i\in[n]\quad \text{ and }\quad t=\frac{1}{d}\prod_{i=1}^{n+1}\left(\frac{c_i}{\lambda^b_{a_i}}\right)^{\lambda^b_{a_i}}
\]
is the unique positive zero of the system. It follows from \Cref{cor:copositivity:interior} that $f$ is copositive if and only if the $t$-coordinate of this solution is larger or equal to 1,  equivalently
\begin{equation}\label{eq:circuit_number}
    d\leq \prod_{i=1}^{n+1}\left(\frac{c_i}{\lambda^b_{a_i}}\right)^{\lambda^b_{a_i}}=:\Theta\,,
\end{equation}
where $\Theta$ is known as the \emph{circuit number}.
\end{example}

\begin{example}[A square and its barycenter]\label[ex]{example:elimination:square}
We illustrate now that \Cref{cor:copositivity:interior} can be exploited to provide closed formulas for copositivity beyond the well-studied case revisited in \Cref{example:circuit:number}. 

Consider sets $\A_+,\A_-\subseteq\R^2$ where $\A_+$ consists of the four vertices of a square and $\A_-$ consists of the barycenter of the square.  
By \Cref{lemma:affine:transformation:Adiscriminant}, copositivity is invariant under affine transformations of the signed support,  so without loss of generality we assume that   $\A_+=\{(0,0),(2,0),(0,2),(2,2)\}$ and $\A_-=\{(1,1)\}$. The  pair $(\A_+,\A_-)$ satisfies the assumptions of \Cref{cor:copositivity:interior}. Any $f\in\mathcal{S}(\A_+,\A_-)$ is of the form
    \[
f =c_0+c_1x_1^{2}+c_2x_2^{2}+c_3x_1^{2}x_2^{2}-c_4x_1x_2
    \]
    with $c=(c_0,\dots,c_4)\in \R_{>0}^{5}.$ 
We consider the height function $h$ with $h(1,1)=1$ and zero otherwise. The entries of the critical system $F$ of $f$ can be regarded as polynomials in the variables $c_0,\dots,c_4,t,x_1,x_2$, and generate an ideal $I\subseteq \R[c_0,\dots,c_4,t,x_1,x_2]$. Using \texttt{Oscar} \cite{OSCAR}, 
we find that the elimination ideal $\widetilde{I} \coloneqq\big(I\colon\langle c_0c_1c_2c_3c_4x_1x_2\rangle^\infty\big) \cap \R[c,t]$ is generated by the polynomial 
\[
g \coloneqq c_4^4t^4-8(c_0c_4^2c_3+c_4^2c_1c_2)t^2+16(c_0^2c_3^2+c_1^2c_2^2-2c_0c_1c_2c_3)\quad  \in \R[c,t]\, .
\]
The four roots of $g$ in $t$ are the square roots of 
    \[
    t_1 =\frac{4(c_0c_3+c_1c_2)+ 8\sqrt{c_0c_1c_2c_3}}{c_4^2}\quad \text{ and } \quad t_2 =\frac{4(c_0c_3+c_1c_2)- 8\sqrt{c_0c_1c_2c_3}}{c_4^2} \, .
    \]
    For positive $c\in \R^5_{>0}$, $t_1$ is positive. By the arithmetic-geometric mean   inequality, $\sqrt{c_0c_1c_2c_3}\leq \tfrac{1}{2} (c_0c_3+c_1c_2)$, so $t_2$ is nonnegative. 
   Therefore, $g$ may have up to two positive roots 
\[ t_+\coloneqq\sqrt{t_1}\quad \text{ and } \quad t_-\coloneqq\sqrt{t_2}\, . \] 
By \Cref{cor:copositivity:interior}, at least one of $t_+,t_-$ must extend to a positive zero of   $F(c\star t^h, x)$ and  $t_*\in\{t_-,t_+\}$. 
At this point, we can assert  that  a signomial $f\in \mathcal{S}(\A_+,\A_-)$ with vector of nonsigned coefficients $c\in \R^5_{>0}$ is copositive if $ t_-\geq 1$, that is, if
\[
c_4^2\leq 4(c_0c_3+c_1c_2)- 8\sqrt{c_0c_1c_2c_3}\, ,
\] 
and is not copositive if $1> t_+$, that is, if
\[
c_4^2> 4(c_0c_3+c_1c_2)+ 8\sqrt{c_0c_1c_2c_3}\, .
\]
After studying copositivity for nonseparable signed supports in \Cref{sec:nonseparable:supports}, we will show in \Cref{example:elimination:square:end} that, in fact, $f$ is copositive if and only if $t_+\geq 1$,  equivalently
\[
    c_4^2\leq 4(c_0c_3+c_1c_2)+ 8\sqrt{c_0c_1c_2c_3}=\big(2\sqrt{c_0c_3}+2\sqrt{c_1c_2}\big)^2\, \quad \Leftrightarrow \quad c_4\leq \sqrt{4c_0c_3}+\sqrt{4c_1c_2}\, . 
    \]
The summands of the last inequality  are the circuit numbers of the 1-dimensional circuits defined by the diagonals of the square. Hence,   $f$ is copositive if and only if $c_4$  is at most the sum of the  circuit numbers of the diagonals.
 \end{example}

\section{Nonseparable signed supports}\label{sec:nonseparable:supports}

In this section we show that a class of signed supports, termed \emph{nonseparable}, satisfy that all copositive signomials admit a SONC decomposition. This is done in \Cref{thm:nonseparable:signed:supports:SONC} in \Cref{subsec:thm_SONC}. 
Before that, we start in \Cref{subsec:geometry:nonseparable} by reviewing the concept of nonseparability, which we learnt from the upcoming work  \cite{Mate&Timo2025},  and focus on the geometric properties of nonseparable signed supports. In \Cref{subsec:SONC:supports}, we discuss properties of the SONC cone and introduce SONC signed supports. In \Cref{subsec:decomposing:critical:system}, the properties of nonseparable signed supports are used together with \Cref{cor:copositivity:interior} to find SONC decompositions. The latter are based on  decompositions of the critical system and extend ideas from \cite{Wang_Nonneg}. 

\subsection{The geometry of nonseparable pairs}\label{subsec:geometry:nonseparable} 
We start by recalling a few concepts from polyhedral geometry that we use in what follows. A  hyperplane in $\R^d$ can be written as 
\[
\mathcal{H}_{v,w}\coloneqq\{x\in\R^d\:\colon\:\langle x,v\rangle+w=0\}
\]
for some $v\in\R^d$ and $w\in\R$, and defines two half-spaces of $\R^d$, 
\[ \mathcal{H}_{v,w}^+\coloneqq\{x\in\R^d\:\colon\:\langle x,v\rangle+w\geq 0\}\quad  \text{and} \quad \mathcal{H}_{v,w}^-\coloneqq\{x\in\R^d\:\colon\:\langle x,v\rangle+w \leq 0\}\, .\] 
For a polytope $P\subseteq\R^d$  and a face $\Gamma$  of $P$, a \term{supporting hyperplane} of $\Gamma$ with respect to $P$ is a hyperplane $\mathcal{H}_{v,w}$ such that $P\cap\mathcal{H}_{v,w}=\Gamma$ and $P\subseteq \mathcal{H}_{v,w}^+$. 

Given a collection of polyhedral complexes $P_1,\dots, P_k$, its \term{common refinement} is the polyhedral complex   $\{G_1\cap\dots\cap G_k\:\colon\:G_i\in 
P_i\}$. The \term{d-cells} of a polyhedral complex $P$ are the polyhedra in $P$ that have dimension $d$.

\begin{definition}[Adapted from \cite{Mate&Timo2025}]\label[def]{def:nonseparable:pairs}
Let $\A_+,\A_-\subseteq \R^n$ be disjoint finite sets such that $ \A_+\cup \A_-$ has dimension $d$ and $\A_-\subseteq\conv(\A_+)$.  
\begin{itemize}
\item We denote by $\mR(\A_+)$  the common refinement of all regular polyhedral subdivisions of $\A_+$. 
We let $\mR_d(\A_+)$ denote the set of $d$-cells of $\mR(\A_+)$. 
\item The pair $(\A_+,\A_-)$ is called \term{nonseparable} if 
\begin{equation}\label{eq:sep}
\text{there exists $D\in\mR_d(\A_+)$ such that $\A_-\subseteq D$}\, .
\end{equation} 
\item If \Cref{eq:sep} does not hold,  then $(\A_+,\A_-)$ is said to be \term{separable}.
\end{itemize}
\end{definition}

Observe that a pair $(\A_+,\A_-)$ is nonseparable if and only if $(\psi(\A_+),\psi(\A_-))$ is nonseparable for any affine transformation $\psi$.

\begin{remark}\label{rk:triangulation}
The set of $d$-cells of  $\mR(\A_+)$ agrees with the set of $d$-cells in the common refinement of \emph{all} polyhedral subdivisions of $\A_+$, that is, of not only the regular ones, and also in the common refinement of all triangulations of $\A_+$. Therefore, one can consider any of these alternative descriptions of  $\mR(\A_+)$ in \Cref{def:nonseparable:pairs}. 
This is true since the $d$-cells in all three common refinements are intersections of $d$-dimensional simplices with vertices in $\A_+$, as each  such simplex is a  $d$-cell of some regular subdivision of $\A_+$. 
This also implies that if $D\in \mR_d(A_+)$, then for all simplices $\Delta$ with vertices in $\A_+$, it holds that either
   \[\relint(D)\subseteq \relint(\Delta)\quad \text{or}\quad \relint(D)\cap \relint(\Delta)=\varnothing \,. \]
\end{remark}

\begin{figure}
    \centering
    \begin{tikzpicture}[scale=1.2]

  \foreach \i in {1,...,5} {
    \coordinate (P\i) at ({90 + (\i-1)*72}:1);
  }

  \foreach \i/\j in {1/2,1/3,1/4,1/5,2/3,2/4,2/5,3/4,3/5,4/5} {
    \draw[name path=E\i\j, black] (P\i) -- (P\j);
  }
  \draw[thick,black] (P1) -- (P3);

  \path[name intersections={of=E13 and E24, by=X1324}];
  \path[name intersections={of=E13 and E25, by=X1325}];
  \path[name intersections={of=E14 and E35, by=X1435}];
  \path[name intersections={of=E14 and E25, by=X1425}];
  \path[name intersections={of=E24 and E35, by=X2435}];

  \fill[blue!20]    (P1) -- (P2) -- (X1325) --  cycle;
  \fill[cyan!20]    (P1) -- (X1425) -- (P5) -- cycle;
  \fill[purple!20]  (P3) -- (P4) -- (X2435) -- cycle;
  \fill[green!20]   (P4) -- (P5) -- (X1435) --  cycle;
  \fill[orange!20]  (P5) -- (P1) --  (X1425) -- cycle;
  \fill[red!20]     (X1324) -- (P2) -- (P3) -- cycle;

  \fill[teal!20]    (P1) --  (X1325) --  (X1425) -- cycle;
  \fill[red!10]     (P2) -- (X1324) -- (X1325) -- cycle;
  \fill[magenta!10] (P3) -- (X1324) -- (X2435) -- cycle;
  \fill[yellow!20]  (P4) -- (X2435) -- (X1435) -- cycle;
  \fill[brown!20]    (P5) -- (X1435) -- (X1425) -- cycle;

  \fill[gray!20] (X1324) -- (X1325) -- (X1425) -- (X1435) -- (X2435) -- cycle;

  \foreach \i in {1,...,5} {
    \draw[red] (P\i) circle (1.5pt);
  }

  \fill[blue] (0.0,0) circle (1.5pt);

  \node at (0.0,-1.1) {nonseparable};
\node at (0.0,-0.2) {\tiny $D$};
\begin{scope}[xshift=3cm]

  \foreach \i in {1,...,5} {
    \coordinate (P\i) at ({90 + (\i-1)*72}:1);
  }

  \foreach \i/\j in {1/2,1/3,1/4,1/5,2/3,2/4,2/5,3/4,3/5,4/5} {
    \draw[name path=E\i\j, black] (P\i) -- (P\j);
  }
  \draw[thick,black] (P1) -- (P3);

  \path[name intersections={of=E13 and E24, by=X1324}];
  \path[name intersections={of=E13 and E25, by=X1325}];
  \path[name intersections={of=E14 and E35, by=X1435}];
  \path[name intersections={of=E14 and E25, by=X1425}];
  \path[name intersections={of=E24 and E35, by=X2435}];

  \fill[blue!20]    (P1) -- (P2) -- (X1325) --  cycle;
  \fill[cyan!20]    (P1) -- (X1425) -- (P5) -- cycle;
  \fill[purple!20]  (P3) -- (P4) -- (X2435) -- cycle;
  \fill[green!20]   (P4) -- (P5) -- (X1435) --  cycle;
  \fill[orange!20]  (P5) -- (P1) --  (X1425) -- cycle;
  \fill[red!20]     (X1324) -- (P2) -- (P3) -- cycle;

  \fill[teal!20]    (P1) --  (X1325) --  (X1425) -- cycle;
  \fill[red!10]     (P2) -- (X1324) -- (X1325) -- cycle;
  \fill[magenta!10] (P3) -- (X1324) -- (X2435) -- cycle;
  \fill[yellow!20]  (P4) -- (X2435) -- (X1435) -- cycle;
  \fill[brown!20]    (P5) -- (X1435) -- (X1425) -- cycle;

  \fill[gray!20] (X1324) -- (X1325) -- (X1425) -- (X1435) -- (X2435) -- cycle;

  \foreach \i in {1,...,5} {
    \draw[red] (P\i) circle (1.5pt);
  }

  \fill[blue] (0.0,0) circle (1.5pt);
  \fill[blue] (0.27,0.13) circle (1.5pt);
   \fill[blue] (0.0,0.3) circle (1.5pt);

  \node at (0.0,-1.1) {nonseparable};
  \node at (0.0,-0.2) {\tiny $D$};
\end{scope}
\begin{scope}[xshift=6cm]

  \foreach \i in {1,...,5} {
    \coordinate (P\i) at ({90 + (\i-1)*72}:1);
  }

  \foreach \i/\j in {1/2,1/3,1/4,1/5,2/3,2/4,2/5,3/4,3/5,4/5} {
    \draw[name path=E\i\j, black] (P\i) -- (P\j);
  }
  \draw[thick,black] (P1) -- (P3);

  \path[name intersections={of=E13 and E24, by=X1324}];
  \path[name intersections={of=E13 and E25, by=X1325}];
  \path[name intersections={of=E14 and E35, by=X1435}];
  \path[name intersections={of=E14 and E25, by=X1425}];
  \path[name intersections={of=E24 and E35, by=X2435}];

  \fill[blue!20]    (P1) -- (P2) -- (X1325) --  cycle;
  \fill[cyan!20]    (P1) -- (X1425) -- (P5) -- cycle;
  \fill[purple!20]  (P3) -- (P4) -- (X2435) -- cycle;
  \fill[green!20]   (P4) -- (P5) -- (X1435) --  cycle;
  \fill[orange!20]  (P5) -- (P1) --  (X1425) -- cycle;
  \fill[red!20]     (X1324) -- (P2) -- (P3) -- cycle;

  \fill[teal!20]    (P1) --  (X1325) --  (X1425) -- cycle;
  \fill[red!10]     (P2) -- (X1324) -- (X1325) -- cycle;
  \fill[magenta!10] (P3) -- (X1324) -- (X2435) -- cycle;
  \fill[yellow!20]  (P4) -- (X2435) -- (X1435) -- cycle;
  \fill[brown!20]    (P5) -- (X1435) -- (X1425) -- cycle;

  \fill[gray!20] (X1324) -- (X1325) -- (X1425) -- (X1435) -- (X2435) -- cycle;
  \foreach \i in {1,...,5} {
    \draw[red] (P\i) circle (1.5pt);
  }

  \fill[blue] (0.325,0.00) circle (1.5pt);
  \fill[blue] (0.265,0.16) circle (1.5pt);

  \node at (0.0,-1.1) {nonseparable};
  \node at (0.0,0.0) {\tiny $D_1$};
  \node at (0.5,0.2) {\tiny $D_2$};
\end{scope}
\begin{scope}[yshift=-2.6cm]

  \foreach \i in {1,...,5} {
    \coordinate (P\i) at ({90 + (\i-1)*72}:1);
  }

  \foreach \i/\j in {1/2,1/3,1/4,1/5,2/3,2/4,2/5,3/4,3/5,4/5} {
    \draw[name path=E\i\j, black] (P\i) -- (P\j);
  }
  \draw[thick,black] (P1) -- (P3);

  \path[name intersections={of=E13 and E24, by=X1324}];
  \path[name intersections={of=E13 and E25, by=X1325}];
  \path[name intersections={of=E14 and E35, by=X1435}];
  \path[name intersections={of=E14 and E25, by=X1425}];
  \path[name intersections={of=E24 and E35, by=X2435}];

  \fill[blue!20]    (P1) -- (P2) -- (X1325) --  cycle;
  \fill[cyan!20]    (P1) -- (X1425) -- (P5) -- cycle;
  \fill[purple!20]  (P3) -- (P4) -- (X2435) -- cycle;
  \fill[green!20]   (P4) -- (P5) -- (X1435) --  cycle;
  \fill[orange!20]  (P5) -- (P1) --  (X1425) -- cycle;
  \fill[red!20]     (X1324) -- (P2) -- (P3) -- cycle;

  \fill[teal!20]    (P1) --  (X1325) --  (X1425) -- cycle;
  \fill[red!10]     (P2) -- (X1324) -- (X1325) -- cycle;
  \fill[magenta!10] (P3) -- (X1324) -- (X2435) -- cycle;
  \fill[yellow!20]  (P4) -- (X2435) -- (X1435) -- cycle;
  \fill[brown!20]    (P5) -- (X1435) -- (X1425) -- cycle;

  \fill[gray!20] (X1324) -- (X1325) -- (X1425) -- (X1435) -- (X2435) -- cycle;

  \foreach \i in {1,...,5} {
    \draw[red] (P\i) circle (1.5pt);
  }

  \fill[blue] (-0.3,-0.6) circle (1.5pt);
  \fill[blue] (0.3,-0.6) circle (1.5pt);
  \fill[blue] (0,-0.82) circle (1.5pt);

  \node at (0.0,-1.1) {nonseparable};
  \node at (0.0,-0.62) {\tiny $D$};
\end{scope}

\begin{scope}[yshift=-2.6cm,xshift=3cm]

  \foreach \i in {1,...,5} {
    \coordinate (P\i) at ({90 + (\i-1)*72}:1);
  }

  \foreach \i/\j in {1/2,1/3,1/4,1/5,2/3,2/4,2/5,3/4,3/5,4/5} {
    \draw[name path=E\i\j, black] (P\i) -- (P\j);
  }
  \draw[thick,black] (P1) -- (P3);

  \path[name intersections={of=E13 and E24, by=X1324}];
  \path[name intersections={of=E13 and E25, by=X1325}];
  \path[name intersections={of=E14 and E35, by=X1435}];
  \path[name intersections={of=E14 and E25, by=X1425}];
  \path[name intersections={of=E24 and E35, by=X2435}];

  \fill[blue!20]    (P1) -- (P2) -- (X1325) --  cycle;
  \fill[cyan!20]    (P1) -- (X1425) -- (P5) -- cycle;
  \fill[purple!20]  (P3) -- (P4) -- (X2435) -- cycle;
  \fill[green!20]   (P4) -- (P5) -- (X1435) --  cycle;
  \fill[orange!20]  (P5) -- (P1) --  (X1425) -- cycle;
  \fill[red!20]     (X1324) -- (P2) -- (P3) -- cycle;

  \fill[teal!20]    (P1) --  (X1325) --  (X1425) -- cycle;
  \fill[red!10]     (P2) -- (X1324) -- (X1325) -- cycle;
  \fill[magenta!10] (P3) -- (X1324) -- (X2435) -- cycle;
  \fill[yellow!20]  (P4) -- (X2435) -- (X1435) -- cycle;
  \fill[brown!20]    (P5) -- (X1435) -- (X1425) -- cycle;

\path[name path=cut1425P3] (X1425) -- (P3);         
\path[name path=inneredge] (X1324) -- (X2435);
\path[name intersections={of=cut1425P3 and inneredge, by=I6}];

  \fill[gray!20] (X1324) -- (X1325) -- (X1425) -- (X1435) -- (X2435) -- cycle;

\draw[gray, very thin] (X1425) -- (P3);

\fill[violet!30] (X1425) -- (X1435) -- (X2435) -- (I6) -- cycle;

\fill[magenta!40] (P3) -- (I6) -- (X2435) -- cycle; 
\draw[red] (X1425) circle (1.5pt);

  \foreach \i in {1,...,5} {
    \draw[red] (P\i) circle (1.5pt);
  }

  \fill[blue] (0.08,0.1) circle (1.5pt);
  \fill[blue] (-0.36,-0.13) circle (1.5pt);

  \node at (0.0,-1.1) {nonseparable};
  \node at (-0.15,0.1) {\tiny $D$};
\end{scope}
\begin{scope}[yshift=-2.6cm,xshift=6cm]

  \foreach \i in {1,...,5} {
    \coordinate (P\i) at ({90 + (\i-1)*72}:1);
  }

  \foreach \i/\j in {1/2,1/3,1/4,1/5,2/3,2/4,2/5,3/4,3/5,4/5} {
    \draw[name path=E\i\j, black] (P\i) -- (P\j);
  }
  \draw[thick,black] (P1) -- (P3);

  \path[name intersections={of=E13 and E24, by=X1324}];
  \path[name intersections={of=E13 and E25, by=X1325}];
  \path[name intersections={of=E14 and E35, by=X1435}];
  \path[name intersections={of=E14 and E25, by=X1425}];
  \path[name intersections={of=E24 and E35, by=X2435}];

  \fill[blue!20]    (P1) -- (P2) -- (X1325) --  cycle;
  \fill[cyan!20]    (P1) -- (X1425) -- (P5) -- cycle;
  \fill[purple!20]  (P3) -- (P4) -- (X2435) -- cycle;
  \fill[green!20]   (P4) -- (P5) -- (X1435) --  cycle;
  \fill[orange!20]  (P5) -- (P1) --  (X1425) -- cycle;
  \fill[red!20]     (X1324) -- (P2) -- (P3) -- cycle;

  \fill[teal!20]    (P1) --  (X1325) --  (X1425) -- cycle;
  \fill[red!10]     (P2) -- (X1324) -- (X1325) -- cycle;
  \fill[magenta!10] (P3) -- (X1324) -- (X2435) -- cycle;
  \fill[yellow!20]  (P4) -- (X2435) -- (X1435) -- cycle;
  \fill[brown!20]    (P5) -- (X1435) -- (X1425) -- cycle;

\path[name path=cut1425P3] (X1425) -- (P3);         
\path[name path=inneredge] (X1324) -- (X2435);      
\path[name intersections={of=cut1425P3 and inneredge, by=I6}];

  \fill[gray!20] (X1324) -- (X1325) -- (X1425) -- (X1435) -- (X2435) -- cycle;

\draw[gray, very thin] (X1425) -- (P3);

\fill[violet!30] (X1425) -- (X1435) -- (X2435) -- (I6) -- cycle;

\fill[magenta!40] (P3) -- (I6) -- (X2435) -- cycle; 
\draw[red] (X1425) circle (1.5pt);

  \foreach \i in {1,...,5} {
    \draw[red] (P\i) circle (1.5pt);
  }

  \fill[blue] (0.0,0) circle (1.5pt);
  \fill[blue] (0.45,0.15) circle (1.5pt);

  \node at (0.0,-1.1) {separable};
  
\end{scope}
\end{tikzpicture}

    \caption{\small Separable and nonseparable pairs $(\A_+,\A_-)$ of dimension $2$. Red circles are in $\A_+$ while solid blue points are in $\A_-$. The cells of $\mR_{2}(\A_+)$ are depicted in different colors. 
    The cells $D,D_1,D_2$ satisfy  \eqref{eq:sep}. For all nonseparable supports but the one depicted at the top right, such a cell is unique.     
    }
    \label{fig:pentagons}
\end{figure}

We now state and prove several geometric properties about nonseparable pairs. These form a quite technical succession of lemmas, but they will be essential in the following sections. Given disjoint finite sets $\A_+,\A_-\subseteq \R^n$   such that $\A_+\cup\A_-$ has dimension $d$, we consider the following set for $D\in \mR_d(\A_+)$:
\begin{align}
\Lambda(\A_+,D) &\coloneqq \{\,\Delta \text{ $d$-dimensional simplex}: \vertices(\Delta)\subseteq \A_+ \text{ and }  \relint(D)\subseteq \relint(\Delta)\,\} \, .\label{eq:lambdaD}
\end{align}
By \Cref{rk:triangulation}, $\Lambda(\A_+,D)\neq \varnothing$. 
If the pair $(\A_+,\A_-)$ is nonseparable, then by definition there exists   $D\in \mR_d(\A_+)$ such that $\A_-\subseteq D$. For this $D$, by taking closures of the interiors,  it follows  that 
\[\A_-\subseteq \Delta \quad \text{ for all } \Delta\in \Lambda(\A_+,D) \, . \]

\begin{example}\label[ex]{example:simplices:nonseparable}
Consider a pair $(\A_+,\A_-)$ as in the pentagon at the top left   of \Cref{fig:pentagons}. Enumerating the vertices of the pentagon by $1$ to $5$ starting at the red circle on top and moving clockwise, and representing each simplex by its vertices, we have that
\begin{align*}
  \Lambda(\A_+,D) & =\big\{\{1,3,4\},\{2,4,5\},\{1,3,5\},\{1,2,4\},\{2,3,5\}\big\}\, .
\end{align*}
There are five additional simplices of dimension $2$ with vertices in $\A_+$, namely $\{1,2,3\}$, $\{2,3,4\}$, $\{3,4,5\}$, $\{1,4,5\}$, and $\{1,2,5\}$. 
For the pentagon at the top right of \Cref{fig:pentagons}, we have 
\[ \Lambda(\A_+,D_2)=\big\{\{1,2,3\}, \{2,4,5\}, \{1,2,4\},\{2,3,5\}\big\} \, . \]
\end{example}

\begin{lemma}\label[lem]{lemma:facet:conditions:nonseparable}
Let $\A_+,\A_-\subseteq \R^n$ be disjoint finite sets such that the pair $(\A_+,\A_-)$ 
has dimension $d$, and assume
$D\in\mR_d(\A_+)$ satisfies $\A_-\subseteq D$. The following statements are equivalent: 
    \begin{enumerate}[label=(\roman*)]
        \item There exists a facet $\Gamma_D$ of $D$ such that $\A_-\subseteq \Gamma_D$.
        \item There exist $\Delta\in\Lambda(\A_+,D)$ and a facet $\Gamma_\Delta$ of $\Delta$ such that $\A_-\subseteq \Gamma_\Delta$.
    \end{enumerate}
\end{lemma}
\begin{proof}
By 
 choosing an affine transformation sending $(\A_+,\A_-)$ to $\{x\in\R^n\colon x_{d+1}=\cdots=x_n=0\}$ if necessary, we can assume that the pair $(\A_+,\A_-)$ is full dimensional.
By definition of $\mR_d(\A_+)$, it holds that
    \[
    D=\bigcap\limits_{\Delta\in \Lambda(\A_+,D)} \Delta \, . 
    \]
Assume that condition (i) holds 
and let $\mathcal{H}_{v,w}$ be a supporting hyperplane of $\Gamma_D$, so $\Gamma_D=D\cap \mathcal{H}_{v,w}$ and
    \[
    \Gamma_D=\bigcap\limits_{\Delta\in \Lambda(\A_+,D)} \Delta\cap \mathcal{H}_{v,w} \, .
    \]
Since $\dim(\Gamma_D)=n-1$ it follows that $\dim(\Delta\cap \mathcal{H}_{v,w})=n-1$ for all $\Delta\in \Lambda(\A_+,D)$. Moreover, as $D\subseteq \mathcal{H}_{v,w}^+$, there must exist $\Delta'\in \Lambda(\A_+,D)$ with $\Delta'\subseteq \mathcal{H}_{v,w}^+$. Hence, $\Gamma_{\Delta'}\coloneqq\Delta'\cap \mathcal{H}_{v,w}$  is a facet of $\Delta'$ satisfying $\A_-\subseteq \Gamma_D\subseteq \Gamma_{\Delta'}$, and thus  condition (ii) holds. 

Assume now that condition (ii) holds. 
Then there exists a supporting hyperplane $\mathcal{H}_{v,w}$ of $\Gamma_\Delta$ such that $\Gamma_\Delta=\Delta\cap \mathcal{H}_{v,w}$ and $\Delta\subseteq \mathcal{H}_{v,w}^+$. As $D\subseteq \Delta$ and $\A_-\subseteq D\cap \mathcal{H}_{v,w}$, $\mathcal{H}_{v,w}$ is also a supporting hyperplane of some face of $D$ containing $\A_-$. In particular, $\A_-$ is contained in a facet of $D$ and condition (i) holds. 
\end{proof}

The next lemma, which relates $\A_+$ and $\Lambda(\A_+,D)$ for nonseparable signed supports, builds on the ideas for the case $\#\A_-=1$   in \cite[Lemma 3.7]{Wang_Nonneg}. An example illustrating the lemma is given in \Cref{fig:correspondence:placeholder}.

\begin{lemma}\label[lem]{lemma:correspondance}
Let $\A_+,\A_-\subseteq \R^n$ be disjoint finite sets such that  $(\A_+,\A_-)$ 
has dimension $d$, and $\A_-\subseteq D$ for some
$D\in\mR_d(\A_+)$. 
Let $\Delta\in \Lambda(\A_+,D)$. 
Then, for each $a\in \A_+\setminus \vertices(\Delta)$, there exists a unique facet $\Gamma_a$ of $\Delta$ such that $conv(\Gamma_a\cup\{a\})\in \Lambda(\A_+,D)$. 

In particular, there is a well-defined and injective map
\begin{align}\label{eq:correspondence:simplices}
    \varphi_\Delta \colon\A_+\setminus \vertices(\Delta) & \longrightarrow \Lambda(\A_+,D) \qquad 
    a  \longmapsto  \conv(\Gamma_a\cup\{a\})\,,
\end{align}
and $\#\Lambda(\A_+,D)\geq \#\A_+ -d$. 
\end{lemma}
\begin{proof}
Fix $a\in \A_+\setminus \vertices(\Delta)$  and let $v_1,\dots,v_{d+1}$ be the   vertices of $\Delta$. Let $\Gamma_i$ be the facet of $\Delta$ opposite to $v_i$. For each $i\in [d+1]$, consider the simplices $P_i\coloneqq\conv(\Gamma_i\cup\{a\})$. We show that exactly one of these simplices is both $d$-dimensional and containing $D$. For $x\in\R^n$ with $x\neq a$, consider the ray from $a$ passing through $x$:
    \[
    R(x)\coloneqq\{tx+(1-t)a\:\colon t\in \R_{\geq 0}\}\,.
    \]
  If $x\in\relint(\Delta)$, then $R(x)$ intersects $\partial \Delta$ at one or two points, depending on the relative position of $a$ with respect to $\Delta$.  In both cases, there exists $y(x)\in R(x)\cap\partial\Delta$ such that $x$ lies in the segment between $y(x)$ and $a$. If $y(x)\in\Gamma_i$ for some $i$, then $x\in P_i$. 
Using this we conclude that
\begin{equation}\label{eq:relint}
\relint(\Delta)=\bigcup\nolimits_{i\in[d+1]}P_{i}^\circ\, \qquad \text{where  } P_{i}^\circ\coloneqq\relint(\Delta)\cap P_i \, . 
\end{equation}
If $P_{i}^\circ\neq \varnothing$, then $\dim(P_i)=d$. Indeed, $\dim(P_i)<d$ if and only if $a\in\Aff(\Gamma_i)$ and, in this case, $\relint(\Delta)\cap P_i=\varnothing$. 
By \eqref{eq:relint}, as $\relint(D)\subseteq \relint(\Delta)$ and $\dim(D)=d$, there must exist $i'$ such that $\relint(D)\cap\relint(P_{i'})\neq\varnothing$ and $P_{i'}^\circ\neq \varnothing$. By \Cref{rk:triangulation}, this means that $\relint(D)\subseteq \relint(P_{i'})$, hence  $P_{i'} \in \Lambda(\A_+,D)$. 

If 
$P_{i}^\circ\neq\varnothing$, we have that
\[
\relint(P_{i}^\circ)=\relint(\Delta)\cap\relint(P_i)
=\{x\in \relint(\Delta)\:\colon y(x)\in \relint(\Gamma_i)\}
\]
and hence the sets $\relint(P_{i}^\circ)$  are pairwise disjoint. As we also have that $\relint(D)\subseteq \relint(P_{i'}^\circ)$, 
we conclude that $P_{i'}$ is the unique simplex among $\{P_i\}_{i\in [d+1]}$ containing $D$. 
This shows that $\varphi_\Delta$ is well defined. 
Injectivity follows immediately as $\varphi_\Delta(a)$ is the only simplex in the image of $\varphi_\Delta$ that has $a$ as a vertex.
For the last statement, the image of $\varphi_\Delta$ contains at least $\#\A_+- (d+1)$ distinct simplices. As $\Delta\in \Lambda(\A_+,D)$ is not in this image, the set $\Lambda(\A_+,D)$ contains at least $\#\A_+-d$ distinct simplices. 
\end{proof}

\begin{figure}
    \centering
    \begin{tikzpicture}[scale=1.5]

  \foreach \i in {1,...,5} {
    \coordinate (P\i) at ({90 + (\i-1)*72}:1);
  }

  \draw[thin] (P1)--(P2)--(P3)--(P4)--(P5)--cycle;

  \begin{scope}[blend group=multiply] 
  \fill[teal!30,  fill opacity=.6] (P1)--(P3)--(P4)--cycle; 
  \fill[magenta!30,         fill opacity=.6] (P2)--(P3)--(P4)--cycle; 
  \fill[green!30,fill opacity=.6] (P5)--(P3)--(P4)--cycle; 
\end{scope}

  \foreach \i in {1,...,5} {
    \draw[red] (P\i) circle (1.5pt);
  }

  \fill[blue] (-0.22,-0.55) circle (1.5pt);
  \fill[blue] (0.22,-0.55) circle (1.5pt);

  \node at (0.,-0.65) {\tiny $D$
};
\node at (0.,-0.95) {\small $\Gamma$
};
\node at (-1.3,0.5) {{\color{magenta!80!black}$\Delta_3$}
};
\node at (0.1,1.3) {{\color{teal!80!black}$\Delta_1$}
};
\node at (1.3,0.5) {{\color{green!70!black}$\Delta_2$}
};
\node at (0.25,1.05) {\small $a_1$
};
\node at (1.2,0.2) {\small $a_2$
};
\node at (0.9,-0.85) {\small $a_3$
};
\node at (-0.9,-0.85) {\small $a_4$
};
\node at (-1.2,0.2) {\small $a_5$
};
  \end{tikzpicture}
   \caption{\small Coloured points represent a nonseparable pair, with solid blue dots in $\A_-$ and red circles in $\A_+$. $D$ is the only  cell of $\mR_2(\A_+)$ containing $\A_-$. The elements of $\Lambda(\A_+,D)=\{\Delta_1,\Delta_2,\Delta_3\}$ are depicted. The map $\varphi_{\Delta_1}$ in \Cref{lemma:correspondance} sends $a_2$ to $\Delta_2=\conv(\{a_2\}\cup\Gamma)$   and $a_5$ to  $\Delta_3=\conv(\{a_5\}\cup\Gamma)$.
   }
    \label{fig:correspondence:placeholder}
\end{figure}

 The following property of nonseparable pairs for which the relative interior of $\conv(\A_+)$ contains some element of $\A_-$ will play an important role in the proof of \Cref{lemma:hessian:nonseparable} in  \Cref{sec:computational}. 

 \begin{lemma}\label[lem]{lemma:dimensionality}
 Let $\A_+,\A_-\subseteq \R^n$ be disjoint finite sets such that  $\A_-\cap\relint(\conv(\A_+))\neq\varnothing$,    $(\A_+,\A_-)$ 
  has dimension $d$, and  $\A_-\subseteq D$ for some $D\in\mR_d(\A_+)$. 

Let $\mathcal{D}\subseteq \Lambda(\A_+,D)$ be a set satisfying the property that every $a\in\A_+$ is the vertex of some simplex in $\mathcal{D}$. Consider $b\in\A_-\cap \relint(\conv(\A_+))$ and for each $\Delta\in\mathcal{D}$, let $\Gamma_{b,\Delta}$ be the unique face of $\Delta$ such that $b\in\relint(\Gamma_{b,\Delta})$. Then, 
 \[
\dim\left(\conv\left(\bigcup\limits_{\Delta\in\mathcal{D}}\Gamma_{b,\Delta}\right)\right)=d\,.
 \]     
 \end{lemma}
\begin{proof}
As the statement holds for a pair $(\A_+,\A_-)$ if and only if it holds for $(\psi(\A_+),\psi(\A_-))$ and any affine transformation $\psi$, we can as usual assume  that the pair $(\A_+,\A_-)$ is full dimensional.
Let $U\coloneqq\bigcup_{\Delta\in \mathcal{D}}\Gamma_{b,\Delta}$, and consider the intersection of  outer normal cones
\[ V\coloneqq\bigcap_{\Delta\in\mathcal{D}}
\{-v\:\colon\:\mathcal{H}_{v,w}\text{ is a supporting hyperplane of }\Gamma_{b,\Delta}\text{ for some }w\in\R\}\, . \]

Suppose for a contradiction that $\dim(\conv(U))<n$. Then, there exists a nonzero vector $v\in \R^n$ such that $\Gamma_{b,\Delta}\subseteq \conv(U)\subseteq\mathcal{H}_{-v,\langle v,b\rangle}$  for all $\Delta\in\mathcal{D}$. 
By changing the sign of $v$ if necessary, there exists $\Delta'\in\mathcal{D}$ such that $\Delta'\subseteq \mathcal{H}_{-v,\langle v, b\rangle}^+$. Any other $\Delta\in\mathcal{D}$ satisfies $\Delta\subseteq \mathcal{H}_{-v,\langle v, b\rangle}^+$ as well, as 
 $\Delta'\cap\Delta$ has dimension $n$ since it contains $D$. Hence $v\in V$. 

Let now $H\coloneqq\mathcal{H}_{-v,\langle v,b\rangle}$. As $b\in\intt(\conv(\A_+))\cap H$ and $\dim(\conv(\A_+))=n$, there exist $a_1,a_2\in\A_+ \setminus H$ such that $a_1\in H^+$ and $a_2\in H^-$. 
By the assumption on $\mathcal{D}$, there exist  $\Delta_1,\Delta_2\in\mathcal{D}$ such that $a_1\in\vertices(\Delta_1)$ and $a_2\in\vertices(\Delta_2)$. As $H$ is a supporting hyperplane of both $\Gamma_{b,\Delta_1}$ and $\Gamma_{b,\Delta_2}$, it follows that $\Delta_1\subseteq H^+$ and $\Delta_2\subseteq H^-$,  so $\Delta_1\cap\Delta_2\subseteq H$ has at most dimension $n-1$. This contradicts $D\subseteq \Delta_1\cap\Delta_2$, since $\dim(D)=n$.
\end{proof}

\subsection{SONC signed supports}\label{subsec:SONC:supports}
In this subsection, we review circuit signomials and the SONC and copositivity cones associated with a given signed support. We revisit conditions from \cite{Wang_Nonneg,murray_Newtonpol,ellwanger} on the signed support  that ensure that the two cones coincide.

Given disjoint finite sets $\A_+,\A_-\subseteq \R^n$, we say that the pair $(\A_+,\A_-)$  is an \term{extended circuit}
if  for $\A=\A_+\cup \A_-$ we either have that $\#\A=\#\A_+=1$  or $\conv(\A)$ is a simplex such that 
$  \A_+=\vertices(\conv(\A))\, . $
In particular $\A_- \subseteq \conv(\A_+)$. 

We say that $(\A_+,\A_-)$  is a  \term{circuit}  if  either $\#\A=\#\A_+=1$ or it is an extended circuit that additionally satisfies  $\#\A_-=1$ and $\A_-\subseteq\relint(\conv(\A_+))$.

A \term{circuit signomial} is a signomial  whose signed support  
is a circuit. The copositivity of a circuit signomial on $\R^n_{>0}$ is characterised in terms of the circuit number \cite{craciun_pantea_koeppl,wolff_nonneg} (cf. \Cref{example:circuit:number}). We say that a signomial $f$ \term{is SONC} (or admits a \textit{SONC decomposition}) if there exist copositive circuit signomials $q_1,\dots,q_r$  such that 
\[f=q_1+\dots+q_r \, . \]

SONC stands for Sum of \textit{Nonnegative} Circuits, while in this work we consider the term copositivity. There is no ambiguity in the nomenclature as nonnegativity (over $\R^n_{>0}$) and copositivity are equivalent notions for signomials. The set of SONC signomials defines the cone
\[
\mathcal{C}_{  \text{{\tiny SONC}},n}\coloneqq\Bigl\{ \sum_{i=1}^rq_i : r\in \Z_{>0},\: q_1,\dots,q_r\text{ copositive circuit signomials  in $\mathcal{S}_n$}\Bigr\}.
\]
This is a subcone of the cone of copositive signomials 
\[
\mathcal{C}_n\coloneqq\big\{f \in \mathcal{S}_n :f\text{ is  copositive}\big\},
\]
called the \term{copositive cone} \cite{motzkin_copositive,Moment_polyomial_optimization}.
We recall the following known result.

\begin{lemma}\label{lem:SONC:singular:zeros}
Let $\A_+,\A_-\subseteq \R^n$ be disjoint finite sets such that $\A_-\subseteq\intt(\conv(\A_+))$ and $(\A_+,\A_-)$ is  full dimensional.
If $f\in \mathcal{S}(\A_+,\A_-)$ is SONC, then 
\[ \#\Sing_{>0}(f)\leq 1\, . \]
\end{lemma}

\begin{proof}
Let $f=q_1+\dots +q_r$ be a SONC decomposition of $f$ and let $\A_i$ be the support of $q_i$. As all $q_i$'s are copositive, $f(x_*)=0$ if and only if $q_i(x_*)=0$ for all $i\in[r]$, and hence  by \Cref{rk:copositive_zeros}, $x_*\in\Sing_{>0}(f)$ if and only if $x_*\in\Sing_{>0}(q_i)$ for all  $i\in [r]$. 

If $\Sing_{>0}(f)\neq \varnothing$, then for each $i$, the set 
$ \Sing_{>0}(q_i)$ is logarithmically affine, that is, 
\[
\Sing_{>0}(q_i)=\text{Exp}(w_i+\Aff(\A_i)^\perp)\quad \text{for some }w_i\in\R^n \, ,
\]
where  Exp is component-wise exponentiation \cite[Theorem 3.1]{algebraic_boundary_SONC}. 
This is due to the relation between copositive circuits and simplicial agiforms (see \cite{algebraic_boundary_SONC,ReznickBruce1989Fdft} for details).  
This implies   that  
\[
\Sing_{>0}(f)=\bigcap_{i=1}^r \Sing_{>0}(q_i)=\bigcap_{i=1}^r\text{Exp}(w_i+\Aff(\A_i)^\perp)=\text{Exp}\left(\bigcap_{i=1}^r(w_i+\Aff(\A_i)^\perp)\right) \, .
\]
 A finite intersection of affine spaces can be either  empty, one point, or contain a line. As $\conv(\A)$ is full dimensional and $\A_-\subseteq\intt(\conv(\A_+))$,  $\Sing_{>0}(f)$ is contained in a  compact subset of $\R^n_{>0}$ 
by \Cref{prop:relatively:compact:set}(ii).
 Since the exponential image of an affine line is not compact, we conclude that $\Sing_{>0}(f)$ is one point when nonempty.
  \end{proof}

The cones $\mathcal{C}_{  \text{{\tiny SONC}},n}$ and $\mathcal{C}_n$ do not agree as illustrated by the Robinson polynomial \cite{Robinson} 
 \[f = 1+x_1^6+x_2^6-(x_1^4x_2^2+x_1^2x_2^4+x_1^4+x_2^4+x_1^2+x_2^2)+3x_1^2x_2^2 \,, \]
 which belongs to $\mathcal{C}_2$ but does not admit a SONC decomposition. However, there exist signed supports for which  any signomial that is copositive also admits a SONC decomposition. 
To formalise this, given disjoint finite sets $\A_+,\A_-\subseteq \R^n$, we introduce the subcones 
\begin{align*}
    \mathcal{C}_{  \text{{\tiny SONC}}}(\A_+,\A_-) & \coloneqq \mathcal{C}_{  \text{{\tiny SONC}},n} \cap \mathcal{S}(\A_+,\A_-),&
    \mathcal{C}(\A_+,\A_-) & \coloneqq  \mathcal{C}_n \cap \mathcal{S}(\A_+,\A_-)\, .
\end{align*}

While the ambient vector space of the cones $\mathcal{C}_n$ and $\mathcal{C}_{  \text{{\tiny SONC}},n}$ has infinite dimension, that  of   $\mathcal{C}(\A_+,\A_-)$ and $\mathcal{C}_{  \text{{\tiny SONC}}}(\A_+,\A_-)$ has finite dimension $\#(\A_+ \cup \A_-)$. 
We say that $(\A_+,\A_-)$ is a \term{SONC signed support} if 
\[\mathcal{C}(\A_+,\A_-)=\mathcal{C}_{  \text{{\tiny SONC}}}(\A_+,\A_-)\, .\] 

Several conditions are known to guarantee that a signed support $(\A_+,\A_-)$  is SONC:
\begin{itemize}
    \item If $\A_+\cup\A_-\subseteq \Z^n$, $\A_+$ is the set of vertices of a simplex,   and $\A_-\subseteq \relint(\conv(\A_+))$ \cite[Corollary 7.5]{wolff_nonneg}. 
    \item If $\A_+\cup\A_-\subseteq\Z^n$, $\#\A_-\leq 1$ and $\A_-\subseteq\conv(\A_+)$ \cite[Theorem 3.9]{Wang_Nonneg}. This result was independently proven in \cite{murray_Newtonpol} in the more general setting of real exponents.
    \item If $(\A_+,\A_-)$ is an extended circuit \cite[Theorem 1.2]{ellwanger}.
\end{itemize}
Additionally, in \cite[Theorem 4.1]{Wang_Nonneg}, Wang provides an additional condition to guarantee that a signed support is SONC that inspired a big part of this work: 
assume that $\A_+\cup\A_-\subseteq\Z^n$, $\A_-\subset\relint(\conv(\A_+))$, that some vertex of $\conv(\A_+)$ is \textit{simple},  i.e. it lies in precisely $\dim(\conv(\A_+))$ many edges of $\conv(\A_+)$, and that $\A_-$ lies in the interior of a halfspace of all hyperplanes  determined by points in $\A_+$. Then $(\A_+,\A_-)$ is a SONC signed support.

Wang conjectured that the assumption of simplicity of some vertex could be dropped. Moreover, although \cite[Theorem 4.1]{Wang_Nonneg} is stated in terms of hyperplanes determined by the points in $\A_+$, all that is used in the proof is that any simplex of maximal dimension with vertices in $\A_+$ either does not intersect $\A_-$ or contains it in its interior, linking to the notion of nonseparability. 

Interestingly, all signed supports such that $\A_-$ is in the interior of $\conv(\A_+)$, and that are known to be SONC, are nonseparable (\Cref{def:nonseparable:pairs}).  However, not all separable signed supports are known to be SONC. For example, the nonseparable supports  in \Cref{fig:pentagons}  with $\A_-$ intersecting the boundary of the cells in $\mR(\A_+)$ are configurations not covered by any of the results above, and in particular by \cite[Theorem 4.1]{Wang_Nonneg}.  The same holds for the configuration in \Cref{fig:Wang:condition}, as it has no simple vertex.

\begin{figure}
    \centering
    \begin{tikzpicture}[
  line join=round, line cap=round
]
\begin{scope}[rotate=180]

\pgfmathtruncatemacro{\N}{4}  
\pgfmathsetmacro{\R}{2.0}      
\pgfmathsetmacro{\H}{2.2}      
\pgfmathtruncatemacro{\NmOne}{\N-1}
\def\dotrad{2pt}            

\pgfmathsetmacro{\Vx}{0.0}
\pgfmathsetmacro{\Vy}{0.0}
\pgfmathsetmacro{\Vz}{-1.0}

\coordinate (T) at (0,0,\H);
\coordinate (B) at (0,0,-\H);
\coordinate (Tb) at (-0.28,0.6,-\H+0.55);
\coordinate (Tbl) at (-0.2,0.3,-\H+0.4);

\foreach \k in {0,...,\NmOne}{
  \pgfmathsetmacro{\ang}{360*\k/\N}
  \coordinate (P\k) at ({\R*cos(\ang)},{\R*sin(\ang)},0);
}

\begin{scope}[blend group=multiply]

  \foreach \k in {0,...,\NmOne}{
    \pgfmathtruncatemacro{\kp}{mod(\k+1,\N)}
    \fill[blue!20!white, fill opacity=.6] (T)--(P\k)--(P\kp)--cycle;
  }

  \foreach \k in {0,...,\NmOne}{
    \pgfmathtruncatemacro{\kp}{mod(\k+1,\N)}
    \fill[cyan!25!white, fill opacity=.6] (B)--(P\kp)--(P\k)--cycle;
  }
\end{scope}

\newif\iffrontTopEdge
\newif\iffrontBottomEdge
\newif\iffrontBaseEdge

\foreach \k in {0,...,\NmOne}{
  \pgfmathtruncatemacro{\kp}{mod(\k+1,\N)}
  \pgfmathtruncatemacro{\km}{mod(\k-1,\N)}

  \pgfmathsetmacro{\bx}{\R*cos(360*\k/\N)}   \pgfmathsetmacro{\by}{\R*sin(360*\k/\N)}   \pgfmathsetmacro{\bz}{0}
  \pgfmathsetmacro{\cx}{\R*cos(360*\kp/\N)}  \pgfmathsetmacro{\cy}{\R*sin(360*\kp/\N)}  \pgfmathsetmacro{\cz}{0}
  \pgfmathsetmacro{\dx}{\R*cos(360*\km/\N)}  \pgfmathsetmacro{\dy}{\R*sin(360*\km/\N)}  \pgfmathsetmacro{\dz}{0}

  \pgfmathsetmacro{\Sone}{ ((\by-0)*(\cz-\H) - (\bz-\H)*(\cy-0))*\Vx + ((\bz-\H)*(\cx-0) - (\bx-0)*(\cz-\H))*\Vy + ((\bx-0)*(\cy-0) - (\by-0)*(\cx-0))*\Vz }

  \pgfmathsetmacro{\Stwo}{ ((\dy-0)*(\bz-\H) - (\dz-\H)*(\by-0))*\Vx + ((\dz-\H)*(\bx-0) - (\dx-0)*(\bz-\H))*\Vy + ((\dx-0)*(\by-0) - (\dy-0)*(\bx-0))*\Vz }
  \frontTopEdgefalse
  \ifdim \Sone pt>0pt \frontTopEdgetrue\fi
  \ifdim \Stwo pt>0pt \frontTopEdgetrue\fi

  \pgfmathsetmacro{\Sbotone}{ ((\cy-0)*(\bz+\H) - (\cz+\H)*(\by-0))*\Vx + ((\cz+\H)*(\bx-0) - (\cx-0)*(\bz+\H))*\Vy + ((\cx-0)*(\by-0) - (\cy-0)*(\bx-0))*\Vz }
  
  \pgfmathsetmacro{\Sbottwo}{ ((\by-0)*(\dz+\H) - (\bz+\H)*(\dy-0))*\Vx + ((\bz+\H)*(\dx-0) - (\bx-0)*(\dz+\H))*\Vy + ((\bx-0)*(\dy-0) - (\by-0)*(\dx-0))*\Vz }
  \frontBottomEdgefalse
  \ifdim \Sbotone pt>0pt \frontBottomEdgetrue\fi
  \ifdim \Sbottwo pt>0pt \frontBottomEdgetrue\fi

  \frontBaseEdgefalse
  \ifdim \Sone pt>0pt \frontBaseEdgetrue\fi
  \ifdim \Sbotone pt>0pt \frontBaseEdgetrue\fi

  \iffrontTopEdge\else     \draw[thin, black!50] (T)--(P\k);\fi
  \iffrontBottomEdge\else  \draw[thin, black!50] (B)--(P\k);\fi
  \iffrontBaseEdge\else    \draw[thin, black!50] (P\k)--(P\kp);\fi
}

\foreach \k in {0,...,\NmOne}{
  \pgfmathtruncatemacro{\kp}{mod(\k+1,\N)}
  \pgfmathtruncatemacro{\km}{mod(\k-1,\N)}
  \pgfmathsetmacro{\bx}{\R*cos(360*\k/\N)}   \pgfmathsetmacro{\by}{\R*sin(360*\k/\N)}   \pgfmathsetmacro{\bz}{0}
  \pgfmathsetmacro{\cx}{\R*cos(360*\kp/\N)}  \pgfmathsetmacro{\cy}{\R*sin(360*\kp/\N)}  \pgfmathsetmacro{\cz}{0}
  \pgfmathsetmacro{\dx}{\R*cos(360*\km/\N)}  \pgfmathsetmacro{\dy}{\R*sin(360*\km/\N)}  \pgfmathsetmacro{\dz}{0}
  \pgfmathsetmacro{\Sone}{ ((\by-0)*(\cz-\H) - (\bz-\H)*(\cy-0))*\Vx + ((\bz-\H)*(\cx-0) - (\bx-0)*(\cz-\H))*\Vy + ((\bx-0)*(\cy-0) - (\by-0)*(\cx-0))*\Vz }
  \pgfmathsetmacro{\Sbotone}{ ((\cy-0)*(\bz+\H) - (\cz+\H)*(\by-0))*\Vx + ((\cz+\H)*(\bx-0) - (\cx-0)*(\bz+\H))*\Vy + ((\cx-0)*(\by-0) - (\cy-0)*(\bx-0))*\Vz }
 
  \pgfmathsetmacro{\Stwo}{ ((\dy-0)*(\bz-\H) - (\dz-\H)*(\by-0))*\Vx + ((\dz-\H)*(\bx-0) - (\dx-0)*(\bz-\H))*\Vy + ((\dx-0)*(\by-0) - (\dy-0)*(\bx-0))*\Vz }
  \pgfmathsetmacro{\Sbottwo}{ ((\by-0)*(\dz+\H) - (\bz+\H)*(\dy-0))*\Vx + ((\bz+\H)*(\dx-0) - (\bx-0)*(\dz+\H))*\Vy + ((\bx-0)*(\dy-0) - (\by-0)*(\dx-0))*\Vz }
  \frontTopEdgefalse
  \ifdim \Sone pt>0pt \frontTopEdgetrue\fi
  \ifdim \Stwo pt>0pt \frontTopEdgetrue\fi
  \frontBottomEdgefalse
  \ifdim \Sbotone pt>0pt \frontBottomEdgetrue\fi
  \ifdim \Sbottwo pt>0pt \frontBottomEdgetrue\fi
  \frontBaseEdgefalse
  \ifdim \Sone pt>0pt \frontBaseEdgetrue\fi
  \ifdim \Sbotone pt>0pt \frontBaseEdgetrue\fi

  \iffrontTopEdge     \draw[line width=0.7pt] (T)--(P\k);\fi
  \iffrontBottomEdge  \draw[line width=0.7pt] (B)--(P\k);\fi
  \iffrontBaseEdge    \draw[line width=0.7pt] (P\k)--(P\kp);\fi
}

\draw[red] (T) circle[radius=\dotrad];
\draw[red] (B) circle[radius=\dotrad];
\foreach \k in {0,...,\NmOne}{ \draw[red] (P\k) circle[radius=\dotrad]; }

\fill[blue] (Tb) circle[radius=\dotrad];
\fill[blue] (Tbl) circle[radius=\dotrad];
\end{scope}

\end{tikzpicture}

    \caption{\small A nonseparable pair $(\A_+,\A_-)$, with hollow red points representing $\A_+$ and solid blue points representing $\A_-$. No vertex of the octahedron is simple. Any signomial with this signed support does not satisfy the assumptions of \cite[Theorem 4.1]{Wang_Nonneg}.}
    \label{fig:Wang:condition}
\end{figure}

The rest of this subsection is devoted to build the ingredients to show in \Cref{thm:nonseparable:signed:supports:SONC} that 
\[
\begin{aligned}
    (\mathcal{A}_+,\mathcal{A}_-) &\text{ nonseparable signed support in } \mathbb{R}^n \\
    &\text{with } \mathcal{A}_-\subseteq \relint( \conv(\mathcal{A}))
\end{aligned}
\quad \implies \quad
\mathcal{C}_{\mathrm{SONC}}(\mathcal{A}_+,\mathcal{A}_-) = \mathcal{C}(\mathcal{A}_+,\mathcal{A}_-)\, . 
\]
In fact, a slightly more general version is shown. 
The criterion for copositivity from \Cref{cor:copositivity:interior} will play a critical role in our proof. The main idea is as follows. For a copositive signomial $f$ with a nonseparable signed support, we consider an associated signomial $f_{t}$ obtained from a height function lifting $\A_-$. For $t_*$ defined as in \Cref{cor:copositivity:interior}, we find a SONC decomposition of $f_{t_*}$. This determines a SONC decomposition of $f_{t}$ for all $t<t_*$, and in particular for $f$, as $1<t_*$ because $f$ is copositive. 
The crucial step in our strategy is to find a SONC decomposition of $f_{t_*}$ when the support is nonseparable. So, we focus  on finding SONC decompositions of signomials with at least one positive singular zero.

\begin{remark}\label[rmk]{rk:tricks:SONC:decomposition}\label[rmk]{rk:affine:nonnegative}
    Let $f\in\mathcal{S}(\A_+,\A_-)$ and $x_*\in \Sing_{>0}(f)$. When finding SONC decompositions,  there are two assumptions we can make without loss of generality.

    \begin{itemize}
        \item We can assume that $x_*=\mathbbm{1}$. Indeed, $f$ and $q(w)\coloneqq f(x_*\star w)$ have the same signed support  and $\mathbbm{1}\in\Sing_{>0}(q)$. Moreover, $q(w)=\sum_{i=1}^r q_i(w)$ is a SONC decomposition of $q$ if and only if $f(w)=\sum_{i=1}^r q_i(x_*^{-1}\star w)$ is a SONC decomposition of $f$. 

For later use, we observe  that by letting $c$ be the nonsigned coefficients of $f$, 
$\mathbbm{1}\in\Sing_{>0}(f)$ if and only if 
\begin{equation}\label{eq:crit:system:f}
    \sum_{a\in \A_+}c_a =\sum_{b\in\A_-}c_b\,, \qquad  \text{and}\qquad 
    \sum_{a\in \A_+}c_a\,a =\sum_{b\in\A_-}c_b\,b\,.
\end{equation}
  
    \item We can assume   that $(\A_+,\A_-)$ is full dimensional. Indeed, by \Cref{lemma:affine:transformation:Adiscriminant} and  observations similar to \Cref{rk:affine:transformation:lower:dimension},   if $\A=\A_+\cup \A_-$  has dimension $d<n$, there exists an affine transformation such that $\psi(\A)$ lies in the subspace $x_{d+1}=\cdots=x_n=0$. With the   natural identification of $\R^{\A}_{>0}$ with $\R^{\psi(\A)}_{>0}$ we have 
    \[ \mathcal{C}_{  \text{{\tiny SONC}}}(\A_+,\A_-)=\mathcal{C}_{  \text{{\tiny SONC}}}(\psi(\A_+),\psi(\A_-)) \,,\quad  \mathcal{C}(\A_+,\A_-)=\mathcal{C}(\psi(\A_+),\psi(\A_-)) \, . \]

\end{itemize}
\end{remark}

\subsection{Decomposing the critical system}\label{subsec:decomposing:critical:system}
This subsection is devoted to finding SONC decompositions of signomials that have positive singular zeros and its signed support is nonseparable. We adapt key ideas appearing in the proof of \cite[Lemma 3.7]{Wang_Nonneg} and rephrase the techniques in terms of the critical system of the signomial we want to decompose. The property that links SONC decompositions and the critical system is given in the following lemma.

\begin{lemma}\label[lem]{lemma:singular:circuit:copositive}
    If $f$ is a circuit signomial with $\Sing_{>0}(f)\neq \varnothing$, then $f$ is copositive.
\end{lemma} 
\begin{proof} By \Cref{lemma:affine:transformation:Adiscriminant} and \Cref{rk:affine:transformation:lower:dimension}, we can assume that $f$ is a full dimensional circuit. In this case, $f$ has exactly one extremal point in $\R^n_{>0}$ which is always a minimum \cite[Proposition 3.3]{wolff_nonneg}. So the existence of a positive singular zero implies that this minimum attains value zero and the signomial is copositive. \end{proof}

In light of \Cref{lemma:singular:circuit:copositive}, we will find a SONC decomposition of a signomial $f$ having a positive singular  zero $x_*\in\Sing_{>0}(f)$ by decomposing the critical system as a sum of critical systems of circuit signomials with $x_*$ as a positive singular zero. We start by writing the critical system in terms of the barycentric coordinates of the points in $\A_-$.

\begin{lemma}\label[lem]{lemma:bary:coord}
Let $\A_+,\A_-\subseteq \R^n$ be disjoint finite sets such that  $\A_-\neq \varnothing$ and $(\A_+,\A_-)$ is a full dimensional extended circuit. 
Let $f\in \mathcal{S}(\A_+,\A_-)$ have vector of nonsigned coefficients $c\in \R^{\A_+\cup \A_-}_{>0}$. 
Then $\mathbbm{1}\in\Sing_{>0}(f)$ if and only if 
\begin{equation}\label{eq:crit:syst:simplex:poly}
    c_a=\sum_{b\in \A_-}\lambda_{a}^{b}c_b\quad \quad \text{ for all }a\in \A_+,
\end{equation}
where $\lambda_{a}^{b}$ is the $a$-th barycentric coordinate of $b$ with respect to $\A_+$. 
\end{lemma}
\begin{proof} 
By hypothesis $\A_+=\{a_1,\dots,a_{n+1}\}$ is an affinely independent set. 
By letting $\A_-=\{b_1,\dots,b_m\}$, we construct the matrix  $B\in \R^{(n+1)\times m}$ 
whose columns contain the barycentric coordinates of  $b_1,\dots,b_m$. 
Then, by \Cref{rk:barycentric},  the reduced row echelon form of the matrix $C$ in \cref{eq:crit:syst:matrix:form}  is
    \[
 C'\coloneqq   \begin{pmatrix}[c|c]
        {\rm Id}_{n+1} &   -B
    \end{pmatrix}\in \R^{(n+1)\times (n+1+m)}\, .  
    \]
As the zeros of the critical system depend on  $C$  only up to linear combinations of its rows, $\mathbbm{1}\in\Sing_{>0}(f)$ if and only if $C' \diag(c)=0$ and the statement follows. 
\end{proof}

 In the next proposition, we determine a SONC decomposition for any singular signomial whose signed support is an extended circuit. 

\begin{proposition}[SONC decomposition for extended circuits]\label[prop]{prop:simplexNP:pos:sing:SONC}
  Let $\A_+,\A_-\subseteq \R^n$ be disjoint finite sets such that $\A_-\neq \varnothing$ and $(\A_+,\A_-)$ is an extended circuit.
  For every $b\in \A_-$, let 
  $\Gamma_{b}$  be the unique face of $\conv(\A_+)$ such that $b\in \relint(\Gamma_{b})$. 

 If $f\in\mathcal{S}(\A_+,\A_-)$ is such that  $\Sing_{>0}(f)\neq \varnothing$, then $f$ admits a SONC decomposition
\[
f=\sum_{b\in \A_-}q_b \, ,
\]
where $q_b$ is a circuit signomial with signed support $(\A_+\cap\Gamma_{b},\{b\})$.
\end{proposition}

\begin{proof}
Write $f=\sum_{a\in\A_+}c_a\, x^a-\sum_{b\in\A_-}c_b\, x^b$ with $c_a>0$ and $c_b>0$ for all $a\in\A_+$ and $b\in \A_-$. By \Cref{rk:tricks:SONC:decomposition}, we assume without loss of generality that $\mathbbm{1}\in\Sing_{>0}(f)$ and that $(\A_+,\A_-)$ is full dimensional. 
As $\mathbbm{1}\in\Sing_{>0}(f)$, with the notation of \Cref{lemma:bary:coord}, we have that \cref{eq:crit:syst:simplex:poly} holds and in particular, for every $a\in \A_+$,   $\lambda_{a}^{b}>0$ for  some $b\in \A_-$. For every $b\in \A_-$, we define
\begin{equation}\label{eq:SONC:decomposition:simplex}
    q_b\coloneqq\sum_{a\in \A_+}\lambda_a^b\, c_b\, x^a - c_b\, x^b=c_b\left(\sum_{a\in \A_+}\lambda_a^bx^a-x^b\right)\, .
\end{equation}
Clearly, $f=\sum_{b\in \A_-}q_b$ by \cref{eq:crit:syst:simplex:poly}. The support of $q_b$ consists of $b$ and 
all $a\in\A_+$ with $\lambda_{a}^{b}>0$. The latter are exactly the vertices of $\Gamma_{b}$, which is a simplex, and hence $q_b$ is a circuit signomial. Finally, applying \Cref{lemma:bary:coord}  to $q_b$,  we find that  $\mathbbm{1}\in\Sing_{>0}(q_b)$  and hence by \Cref{lemma:singular:circuit:copositive}, $q_b$ is copositive. 
\end{proof}

Our next goal is to show that signomials with nonseparable signed support that have a positive singular zero decompose as a sum of signomials whose signed support is an extended circuit. Hence by \Cref{prop:simplexNP:pos:sing:SONC}, these signomials admit a SONC decomposition.
We show in \Cref{{lemma:decomposition:from:nonnegative:solution},lemma:nonneg:sol:nonseparable:signomial}
that this holds if the positive singular zero is $\mathbbm{1}$, and the general case is treated in \Cref{thm:nonseparable:singular:polynomial:SONC}.

\begin{definition}\label[def]{def:Z}
Let $\A_+,\A_-\subseteq \R^n$  be disjoint finite sets such that $\A=\A_+\cup\A_-$ is full dimensional. 
Let $\Lambda$ be a set of full dimensional simplices  with vertices in $\A_+$ satisfying $\A_-\subseteq \Delta$  for all $\Delta \in\Lambda$. Given  $c\in \R^\A_{>0}$, we let 
$\mathcal{Z}(\Lambda,c)\subseteq \R^{\Lambda}_{\geq 0}$ be  the set of solutions  $\delta\in \R^{\Lambda}_{\geq 0}$ to the linear system
\begin{equation}\label{eq:def:Z}
c_a=\sum\limits_{\substack{\Delta\in \Lambda \\ a\in \vertices(\Delta)}}\delta_\Delta\left(\sum\limits_{b\in \A_-}\lambda_{a,\Delta}^b\, c_b\right)\qquad\text{ for }a\in \A_+\,, 
\end{equation}
 where $\lambda_{a,\Delta}^b\in [0,1]$ are the barycentric coordinates of $b\in \A_-$ with respect to $\vertices(\Delta)$.
\end{definition}

Observe that if $\mathcal{Z}(\Lambda,c)\neq \varnothing$, then  every $a\in\A_+$ has to be a vertex of some simplex $\Delta\in \Lambda$.

\begin{lemma}\label[lem]{lemma:decomposition:from:nonnegative:solution}
With the assumptions and notation in \Cref{def:Z}, for $\Delta\in \Lambda$, we denote by $\Gamma_\Delta$  the smallest face of $\Delta$ containing $\A_-$. Let $f\in \mathcal{S}(\A_+,\A_-)$ have vector of nonsigned coefficients $c\in \R^{\A}_{>0}$ and satisfy  $\mathbbm{1}\in \Sing_{>0}(f)$. 
Then 
the following statements hold:
\begin{enumerate}[label=(\roman*)]
    \item  $\mathcal{Z}(\Lambda,c)\subseteq\mathcal{H}\coloneqq\{\delta\in \R^{\Lambda}_{\geq 0} : \sum_{\Delta\in \Lambda} \delta_\Delta=1\}$.
    \item If $\delta\in \mathcal{Z}(\Lambda,c)$,  then 
    for each $\Delta\in \Lambda$ such that $\delta_\Delta>0$, the signomial
        \[
q_\Delta=\delta_\Delta \left(\sum_{a\in\vertices(\Gamma_\Delta)}\Big(\sum_{b\in\A_-}\lambda_{a,\Delta}^b c_b\Big)x^a-\sum_{b\in\A_-}  c_b\, x^b \right)
    \]
satisfies $\mathbbm{1}\in \Sing_{>0}(q_\Delta)$ and its signed support $(\vertices(\Gamma_\Delta),\A_-)$ is an extended circuit. Furthermore, 
\[f=\sum_{\Delta\in
\Lambda\colon \delta_\Delta>0}q_\Delta\] and hence $f$ is SONC.
\end{enumerate}
\end{lemma}
\begin{proof} 
We start by showing part (i). As $\mathbbm{1}\in\Sing_{>0}(f)$, we have that $\sum_{a\in\A_+}c_a=\sum_{b\in\A_-}c_b$ by \eqref{eq:crit:system:f}. 
Then, for $\delta\in \mathcal{Z}(\Lambda,c)$, we obtain
 \begin{align*}
     0\neq \sum_{a\in \A_+}c_a &=\sum_{a\in \A_+}\left(\sum_{\Delta\in \Lambda\colon a\in \vertices(\Delta)}\delta_\Delta\left(\sum_{b\in \A_-}\lambda_{a,\Delta}^bc_b\right)\right) \\
     &=\sum_{\Delta\in \Lambda}\delta_\Delta\left(\sum_{b\in \A_-}\left(\sum_{a\in \vertices(\Delta)}\lambda_{a,\Delta}^b\right)c_b\right)=\left(\sum_{\Delta\in \Lambda}\delta_\Delta\right)\left(\sum_{a\in\A_+}c_a\right)\,,
 \end{align*}
 from where we deduce part (i). 
 
To show part (ii), as $\lambda_{a,\Delta}^b>0$ if $a\in \vertices(\Gamma_\Delta)$, the signed support of $q_\Delta$   is as given and is an extended circuit. \Cref{lemma:bary:coord} immediately gives that 
 $\mathbbm{1}\in \Sing_{>0}(q_\Delta)$. Hence by \Cref{prop:simplexNP:pos:sing:SONC}, each $q_\Delta$ is SONC. Using \cref{eq:def:Z} and part (i), we obtain that
\begin{align*}
    \sum_{\Delta\in \Lambda\colon \delta_\Delta>0}q_\Delta &= \sum_{\Delta\in \Lambda\colon \delta_\Delta>0} \left(\sum_{a\in\vertices(\Gamma_\Delta)}\delta_\Delta  \left(\sum_{b\in\A_-}\lambda_{a,\Delta}^b c_b\right)x^a-\sum_{b\in\A_-}  \delta_\Delta  c_b\,x^b \right)\\
     &= \sum_{a\in \A_+} \left(\sum\limits_{\Delta\in \Lambda \colon a\in\vertices(\Gamma_\Delta)}\delta_\Delta\left(\sum_{b\in\A_-}\lambda_{a,\Delta}^b c_b\right)x^a\right)-\sum_{b\in\A_{-}}c_b\,x^b  \\
      &=\sum_{a\in \A_+} c_ax^a-\sum_{b\in\A_{-}}c_b\,x^b  = f\,.
\end{align*}
 In particular $f$ is SONC as well, and part (ii) is shown. 
 \end{proof}

\begin{lemma}\label[lem]{lemma:nonneg:sol:nonseparable:signomial}
 Let $\A_+,\A_-\subseteq \R^n$ be disjoint finite sets such that $\A=\A_+\cup \A_-$ is full dimensional. Assume  $\A_-\subseteq D$ for some  $D\in\mR_n(\A_+)$.
   Let $f\in \mathcal{S}(\A_+,\A_-)$ have vector of nonsigned coefficients $c\in \R^{\A}_{>0}$ and  satisfy  $\mathbbm{1}\in \Sing_{>0}(f)$. 
  By letting $\Lambda\coloneqq\Lambda(\A_+,D)$, c.f. \eqref{eq:lambdaD}, the set $\mathcal{Z}(\Lambda,c)$ in \Cref{def:Z} satisfies
 \[\mathcal{Z}(\Lambda,c)  \neq\varnothing \, .  \]
 In particular, $f$ is SONC.
\end{lemma}

\begin{proof}
The last statement of the lemma follows from \Cref{lemma:decomposition:from:nonnegative:solution} once we have shown that  $\mathcal{Z}(\Lambda,c)  \neq\varnothing$. 
Let $\lambda_{a,\Delta}^b\in [0,1]$ be the barycentric coordinates of $b\in \A_-$ with respect to $\vertices(\Delta)$ for $\Delta\in \Lambda$.
If $\#\Lambda=1$, then $\#\A_+=n+1$  by \Cref{lemma:correspondance}, and hence $(\A_+,\A_-)$ is an extended circuit. The statement follows from \cref{lemma:bary:coord}. Assume thus that $\#\Lambda>1$ and hence $\#\A_+>n+1$.

We denote by $M=(m_{a,\Delta})\in \R^{\A_+\times \Lambda}$ the coefficient matrix of \cref{eq:def:Z}, where
\[
m_{a,\Delta}=\begin{cases}
    \sum\limits_{b\in \A_-}\lambda_{a,\Delta}^b\, c_b & \text{if }  a\in \vertices(\Delta)\\
    0 & \text{otherwise,}
\end{cases}
\]
such that 
 the system \eqref{eq:def:Z} becomes $M\delta=c^+$ with $c^+=\{c_a\}_{a\in\A_+}$. For every $\Delta\in \Lambda$, 
 let $M_\Delta$ be the matrix obtained from $M$ by removing the  $\Delta$-column and all rows  for which $m_{a,\Delta}\neq 0$. 
Then the subsystem of \eqref{eq:def:Z} obtained by removing the equations involving the variable $\delta_\Delta$ takes the form 
\[
M_\Delta\hat{\delta}=c^+_\Delta\, , \qquad \hat{\delta} \in \R^{\Lambda\setminus \{\Delta\}}\, , 
\]
 where $c^+_\Delta$ is obtained from $c^+$ in the obvious way. By \cite[Lemma 3.4]{Wang_Nonneg}, if we show that
\begin{enumerate}
    \item $\rank(M)>1$,
    \item $\rank(M_\Delta)=\rank(M)-1$ for all $\Delta\in \Lambda$,
    \item $M\delta=c^+$ admits a solution in $\R^\Lambda$,
    \item $M_\Delta\hat{\delta}=c^+_\Delta$ has a  solution in $\R^{\Lambda\setminus \{\Delta\}}_{\geq 0}$ for all $\Delta\in \Lambda$,
\end{enumerate}
then $\mathcal{Z}(\Lambda,c)\neq\varnothing$. We distinguish two cases depending on how $\A_-$ sits in $D$.

\smallskip
\noindent 
\textbf{Case 1. } We start by assuming that $\A_-$ is not contained in any facet of $D$ and verify that conditions (1)-(4) hold. By \Cref{lemma:facet:conditions:nonseparable}, $\A_-$ is not contained in a facet of any  $\Delta\in \Lambda$. This implies that $m_{a,\Delta}>0$ if and only if   $a\in \vertices(\Delta)$, since $\lambda_{a,\Delta}^b>0$ for some $b\in \A_-$. Therefore, $M_\Delta$ has $\#\A_+-(n+1)$ rows.

Consider  the injective map $\varphi_\Delta$ in \Cref{lemma:correspondance} and recall that every simplex in its image has exactly one vertex that is not a vertex of $\Delta$. 
An important characteristic of $M_\Delta$    is then that
\begin{equation} \label{eq:MDelta}\small
    \text{for every $a\in \A_+\setminus \im(\varphi_\Delta)$, the only nonzero entry of the $\varphi_\Delta(a)$-column of $M_\Delta$ is $m_{a,\Delta}$}\,. 
\end{equation}
Using this, the $\#\A_+-(n+1)$ columns of $M_\Delta$  indexed by the simplices in  $\im(\varphi_\Delta)$ are linearly independent, and hence
\[\rank(M_\Delta)=\#\A_+-(n+1)\, .\] 
Any nonzero linear combination of the columns in $\im(\varphi_\Delta)$ leads to a vector with nonzero $a$-coordinate for some $a\in\A_+\setminus\vertices(\Delta)$. As the $\Delta$-column of $M$ is zero at all such entries, we get that the columns of $M$ indexed by the simplices $\{\Delta\}\cup\im(\varphi_\Delta)$ are linearly independent. Therefore, it must hold that 
\[\rank(M)\geq \rank(M_\Delta)+1=\#\A_+-n\,.\] 
To show that $\rank(M)=\#\A_+-n$, 
it is enough to show that $\im(M)$ is contained in a vector space of dimension $\#\A_+-n$. 
To this end, consider the matrix $N\in \R^{[n]\times \A_+}$ with  $a$-column equal to the vector $a-\bar{b}$, where 
\[
\bar{b}=\frac{1}{\sum\limits_{b\in \A_-}c_b}\left(\sum_{b\in \A_-}c_b\,  b\right).
\]
Since $\bar{b}\in \conv(\A_+)$ and $\conv(\A_+)$ is $n$-dimensional, we have $\rank(N)=n$. For any $\Delta\in \Lambda$, using the definition of  barycentric coordinates, the $\Delta$-column of the product $NM\in \R^{[n]\times \Lambda}$ is
\begin{align*}
    \sum_{a\in \A_+} &\left(a-\frac{1}{\sum\limits_{b\in \A_-}c_b} \left(\sum_{b\in \A_-}c_b\:b\right) \right) m_{a,\Delta}
    \\
    & =\sum_{a\in \vertices(\Delta)}\left(a-\frac{1}{\sum\limits_{b\in \A_-}c_b}\left(\sum_{b\in \A_-}c_b \:b\right)\right)\left(\sum_{b\in \A_-}\lambda_{a,\Delta}^b\:c_b\right)\\
    &=\left(\sum_{b\in \A_-}\sum_{a\in \vertices(\Delta)}\lambda_{a,\Delta}^b\:c_b\:a\right)-\frac{\sum\limits_{b\in\A_-}\left(\sum\limits_{a\in \vertices(\Delta)}\lambda_{a,\Delta}^b\right)c_b}{\left(\sum\limits_{b\in\A_-}c_b\right)}\left(\sum_{b\in\A_-}c_b\:b\right)\\
    & =\left(\sum_{b\in \A_-}\sum_{a\in \vertices(\Delta_k)}\lambda_{a,\Delta}^b\:c_b\:a\right)-\left(\sum_{b\in\A_-}c_b\:b\right)
    =\sum_{b\in \A_-}c_b\left(\sum_{a\in \vertices(\Delta)}\lambda_{a,\Delta}^b\:a-b\right)=0 \, .
\end{align*}
 This shows that $NM=0$, so $\im(M)\subseteq \ker(N)$.  As $\dim(\ker(N))=\#\A_+-n$, we conclude that $\im(M)=\ker(N)$ and hence $\rank(M)= \#\A_+-n$. In particular, conditions (1) and (2) hold.

To show that (3) holds, it is enough to show that $N c^+=0$. As  $\mathbbm{1}\in \Sing_{>0}(f)$,  this follows from a simple computation using \cref{eq:crit:system:f}: 
\begin{align*}
    N c^+ & =\sum_{a\in \A_+}c_a\left(a-\frac{1}{\sum\limits_{b\in \A_-}c_b} \left(\sum_{b\in \A_-}c_b\,  b\right)\right)=\sum_{a\in \A_+}c_a\, a-\frac{\left(\sum\limits_{a\in \A_+}c_a\right)}{\left(\sum\limits_{b\in \A_-}c_b\right)}\left(\sum_{b\in \A_-}c_b\,  b\right)\\
    & = \sum\limits_{a\in\A_+} c_a\, a  - \sum\limits_{b\in \A_-} c_b\,  b =0\, .
\end{align*}
  Finally, for (4), let $\Delta\in \Lambda$
 and  $\hat{\delta}$ be obtained by removing the $\Delta$-entry of the vector $\delta'$ given by 
 \[ \delta'_{\Delta'} = 
\begin{cases}
 0\quad & \text{if } \varnothing=\varphi^{-1}_\Delta(\Delta')  
 \\
   \frac{c_a}{\left(\sum_{b\in \A_-}\lambda_{a,\Delta'}^b c_b\right)} >0 \quad &  
   \text{if } a=\varphi_\Delta^{-1}(\Delta')\,.
\end{cases}
\]
By \eqref{eq:MDelta} and the definition of $M_{a,\Delta'}$, we readily obtain that
$\hat{\delta}$ is a nonnegative solution to $M_\Delta \hat{\delta}=c_\Delta^+$ and (4) holds. 

As the conditions (1)-(4) hold, we conclude  that $\mathcal{Z}(\Lambda,c)\neq\varnothing$ by \cite[Lemma 3.4]{Wang_Nonneg}.

\smallskip
\noindent
\textbf{Case 2. } Assume now that $\A_-\subseteq \Gamma$ for a facet  $\Gamma$ of $D$. By \Cref{lemma:facet:conditions:nonseparable}, $\A_-$ is contained in a facet of some simplex in $\Lambda(\A_+,D)$. We note that the argument of Case 1 cannot be directly applied, as it is no longer true that $m_{a,\Delta}>0$ if and only if $a\in \vertices(\Delta)$. To bypass this problem, we apply a small perturbation to one element of $\A_-$ to reduce to Case 1.
 
Fix any $\Delta\in \Lambda(\A_+,D)$  and choose $b'\in\A_-$ and $v\in \R^n$ such that $b+v\in \intt(D)$. For $\varepsilon\in [0,1]$,  let $b'(\varepsilon)\coloneqq b'+\varepsilon v$ and let $c^+(\varepsilon)\in\R^{\A_+}$ be defined by
\begin{equation}\label{eq:def:c+}
    c^+(\varepsilon)_a\coloneqq\begin{cases}
    c_a+(\lambda_{a,\Delta}^{b'(\varepsilon)}-\lambda_{a,\Delta}^{b'})c_{b'} &\text{  if  }a\in \vertices(\Delta) \\
    c_a&\text{  if  }a\in \A_+\setminus\vertices(\Delta_1)\,.
\end{cases}
\end{equation}
By scaling $v$ if necessary, we can assume that $c^+(\varepsilon)_a>0$ for all $a\in \A_+$ and all $\varepsilon\in [0,1]$. Consider the signomial 
\[
f^\varepsilon\coloneqq\sum_{a\in\A_+}c^+(\varepsilon)_a\, x^a-\sum_{b\in \A_-\setminus\{b'\}}c_b\, x^b-c_{b'}x^{b'(\varepsilon)}\qquad \varepsilon\in [0,1]\,.
\]
For $\varepsilon=0$, we recover the original $f$. 
By construction of $c^+(\varepsilon)$, using that $\mathbbm{1}\in\Sing_{>0}(f)$, \cref{eq:crit:system:f} gives that $\mathbbm{1}\in\Sing_{>0}(f^\varepsilon)$. Write the system  \cref{eq:def:Z} associated with $f^\varepsilon$ as
\begin{equation}\label{eq:linear:system:perturbed}
    M(\varepsilon)\,\delta=c^+(\varepsilon)\, .
\end{equation}

As $b'(\varepsilon)\in\intt(D)$ for all $\varepsilon\in(0,1]$, it follows from Case 1 that  \cref{eq:linear:system:perturbed} admits a nonnegative solution for all $\varepsilon\in(0,1]$. As both $M(\varepsilon)$ and $c^+(\varepsilon)$ depend continuously on the barycentric coordinates of $b'(\varepsilon)$, we have a continuous map
\begin{align*}
    \phi:[0,1]\times \mathcal{H}\cap \R^\Lambda_{\geq 0}
    &\longrightarrow\R^{\A_+}\\
    (\varepsilon,\delta) & \longmapsto M(\varepsilon)  \delta - c^+(\varepsilon)\, ,
\end{align*}
where $\mathcal{H}$ is as in \Cref{lemma:decomposition:from:nonnegative:solution}(i). 
We conclude that the set $\phi^{-1}(0)$ is closed. Given a decreasing sequence $\{\varepsilon_k\}_{k\in\N}\subseteq (0,1]$ converging to 0, by solving \cref{eq:linear:system:perturbed} we obtain a sequence $\{\delta(\varepsilon_k)\}_{k\in\N}\subseteq \mathcal{H}\cap\R^\Lambda_{\geq 0}$ with $\phi(\varepsilon_k,\delta(\varepsilon_k))=0$ for every $k\in\N$. As $\mathcal{H}\cap\R^\Lambda_{\geq 0}$ is compact, there exists a converging subsequence $\{\delta(\varepsilon_{i_k})\}$ with limit $\tilde{\delta}\in \mathcal{H}\cap\R^\Lambda_{\geq 0}$. As $\phi^{-1}(0)$ is closed, we have $\lim_{k\to\infty}(\varepsilon_{i_k},\delta(\varepsilon_{i_k}))=(0,\tilde{\delta})\in \phi^{-1}(0)$.  Hence, $\tilde{\delta}$ is a nonnegative solution of the system \eqref{eq:def:Z} associated with $f$.
\end{proof}

\subsection{Main theorems on nonseparability and SONC decompositions}\label{subsec:thm_SONC}
We have now all the ingredients to show the main results of this work. Recall the definition of the signed discriminant $\nabla_{>0}(\A_+,\A_-)$ in \eqref{eq:DegSignedDiscriminant}, and that we identify a signomial $f$ with its vector of nonsigned coefficients $c$. So we may simply write $f\in \nabla_{>0}(\A_+,\A_-)$.

\begin{theorem}\label[thm]{thm:nonseparable:singular:polynomial:SONC}
 Let $\A_+,\A_-\subseteq \R^n$ be disjoint finite sets such that $(\A_+,\A_-)$ is  nonseparable.
If $f\in \mathcal{S}(\A_+,\A_-)$ satisfies $\Sing_{>0}(f)\neq\varnothing$, then $f$ is SONC. In other words, 
 \[\nabla_{>0}(\A_+,\A_-)\subseteq \mathcal{C}_{\text{{\tiny \rm SONC}}}(\A_+,\A_-)\,.\]
\end{theorem}
\begin{proof}
By \Cref{rk:tricks:SONC:decomposition}, we can assume that $\A=\A_+\cup \A_-$ is full dimensional and that $\mathbbm{1}\in\Sing_{>0}(f)$. \Cref{lemma:decomposition:from:nonnegative:solution,lemma:nonneg:sol:nonseparable:signomial} give  that $f$ is SONC. 
\end{proof}

Next, we show that nonseparable signed supports are SONC, by combining 
\Cref{cor:copositivity:interior} with \Cref{thm:nonseparable:singular:polynomial:SONC}.

\begin{theorem}[Nonseparable supports are SONC]\label[thm]{thm:nonseparable:signed:supports:SONC}
Let $\A_+,\A_-\subseteq \R^n$ be disjoint  nonempty finite sets such that  $\A_-\subseteq \relint(\Gamma)$ for a face $\Gamma$ of $\conv(\A_+)$. 
If $(\A_+^\Gamma,\A_-^\Gamma)$ is nonseparable,  then
\[ \mathcal{C}_{\text{{\tiny \rm SONC}}}(\A_+,\A_-)=\mathcal{C}(\A_+,\A_-)\, .\]
\end{theorem}
\begin{proof}
The inclusion $\mathcal{C}_{\text{{\tiny SONC}}}(\A_+,\A_-)\subseteq \mathcal{C}(\A_+,\A_-)$   holds by definition. For the reverse inclusion, we show that for any $f\in \mathcal{S}(\A_+,\A_-)$ 
it holds that 
\[ f\in \mathcal{C}(\A_+,\A_-) \ \Rightarrow \ f^\Gamma\in \mathcal{C}(\A_+^\Gamma,\A_-^\Gamma)
\ \Rightarrow \  f^\Gamma\in \mathcal{C}_{\text{{\tiny \rm SONC}}}(\A_+^\Gamma,\A_-^\Gamma)  \ \Rightarrow \   f\in \mathcal{C}_{\text{{\tiny \rm SONC}}}(\A_+,\A_-) \,.
\]
The first implication is a consequence of \Cref{lemma:copositivity:signs:truncations}. The third implication follows as  positive monomials are by definition copositive circuits. Hence, all we need is to show the second implication. 
In particular,  we can  simply assume that $\Gamma=\conv(\A_+)$ and 
$f\in \mathcal{C}(\A_+,\A_-)$. Moreover, using \Cref{rk:affine:nonnegative} we restrict to the case where 
$\A=\A_+\cup\A_-$ is full dimensional. 

Therefore, we consider a nonseparable and full dimensional signed support $(\A_+,\A_-)$  such that $\A_-\subseteq \intt(\conv(\A_+))$, and let $f\in \mathcal{C}(\A_+,\A_-)$. If $f\in \nabla_{>0}(\A_+,\A_-)$, then $f\in \mathcal{C}_{\text{{\tiny SONC}}}(\A_+,\A_-)$  by \Cref{thm:nonseparable:singular:polynomial:SONC}. If $f\notin \nabla_{>0}(\A_+,\A_-)$, then let $c\in\R^{\A}_{>0}$ be its vector of nonsigned coefficients and consider a height function $h\colon \A\to\R$ lifting $\A_-$. For $f_t$ as in \cref{eq:Viro}, using that $f$ is copositive it follows from \Cref{cor:copositivity:interior} that there exists $t_*>1$ for which $f_{t_*}$ lies in $\nabla_{>0}(\A_+,\A_-)$. We can assume that $\mathbbm{1}\in\Sing_{>0}(f_{t_*})$ by \Cref{rk:tricks:SONC:decomposition}. 

As the signed support is nonseparable, we can choose $D\in \mR_d(\A_+)$  such that $\A_-\subseteq D$. We are now under the hypotheses of \Cref{lemma:nonneg:sol:nonseparable:signomial}, and therefore there exists 
$\delta$ solving \cref{eq:def:Z}. 
For all $\Delta\in \Lambda(\A_+,D)$ such that $\delta_\Delta>0$, we define  
\[
    q_{\Delta,b,t}(x)\coloneqq \delta_\Delta c_b\left(\sum_{a\in\vertices(\Delta)}t_*^{h(b)}\lambda_{a,\Delta}^b x^a-t^{h(b)}x^b\right), 
    \]
where  $\lambda_{a,\Delta}^b$ denotes as usual the barycentric coordinate of $b$ in $\Delta$ for $a\in \vertices(\Delta)$.  By \Cref{lemma:decomposition:from:nonnegative:solution},
\[
f_{t_*}(x)=\sum_{\delta_\Delta>0}\ \sum_{b\in\A_-}q_{\Delta,b,t_*}(x)
\]
is a SONC decomposition of $f_{t_*}$. 
Using that $\sum_{\delta_\Delta>0}\delta_\Delta=1$ by \Cref{lemma:decomposition:from:nonnegative:solution}, 
 we obtain that
    \begin{align*}
    f(x)& =f_{t_*}(x) + \sum_{b\in\A_-} ( t^{h(b)}_*-1)c_b x^b =
   \sum_{\delta_\Delta>0} \ \sum_{b\in\A_-} q_{\Delta,b,t_*}(x)
 + \sum_{\delta_\Delta>0} \delta_\Delta \sum_{b\in\A_-}( t^{h(b)}_*-1)c_b x^b \\ 
        &=
    \sum_{\delta_\Delta>0}\ \sum_{b\in\A_-}q_{\Delta,b,1}(x)\,.
    \end{align*}
As $q_{\Delta,b,1}(x)>q_{\Delta,b,t_*}(x)$  if $\delta_\Delta>0$,   $x\in\R^{n}_{>0}$, and $b\in \A_-$, this shows that $f$ admits a SONC decomposition and hence $f\in \mathcal{C}_{\text{{\tiny SONC}}}(\A_+,\A_-)$ as desired.
\end{proof}

\Cref{thm:nonseparable:singular:polynomial:SONC}  has an important consequence for signomials in the setting of  \Cref{thm:criterion:copositivity}, namely that the set $T$ in the theorem has exactly one element. Recall the definitions surrounding \cref{eq:truncated}.

\begin{theorem}\label[thm]{thm:1solution}
Let $\A_+,\A_-\subseteq \R^n$ be disjoint nonempty finite sets such that  $\A_-\subseteq \relint(\Gamma)$ for a face $\Gamma$ of $\conv(\A_+)$. If $(\A_+^\Gamma,\A_-^\Gamma)$ is nonseparable, then the set $T$ in \Cref{thm:criterion:copositivity} has cardinality $1$.
\end{theorem}
\begin{proof} By \Cref{thm:criterion:copositivity}, $T$ is nonempty, and $t\in T$ if and only if $f_t^\Gamma\in \nabla_{>0}(\A_+^\Gamma,\A_-^\Gamma)$, since for any other face $\Gamma'\neq\Gamma$ of $\conv(\A_+)$ we have $\A_-\cap\relint(\Gamma')=\varnothing$. 
By \Cref{rk:affine:transformation:lower:dimension} we can assume 
 that $\A^\Gamma=\A_+^\Gamma\cup\A_-^\Gamma$ is full dimensional.  Suppose that $T$ contains two elements $t_*<t_{**}$. Then $f^\Gamma_{t_{**}}$ attains negative values by \Cref{thm:criterion:copositivity}, which contradicts  that $f^\Gamma_{t_{**}}$ should be copositive as $f^\Gamma_{t_{**}}\in\nabla_{>0}(\A_+^\Gamma,\A_-^\Gamma)\subseteq \mathcal{C}_{\text{{\tiny \rm SONC}}}(\A_+^\Gamma,\A_-^\Gamma)$  by \Cref{thm:nonseparable:singular:polynomial:SONC}.
\end{proof}

We conclude this section with a description of the boundary of the copositivity cone for nonseparable signed supports. We 
rely on \cite[Proposition 2.6 and 2.7]{ToricPolya}, which is stated for Laurent polynomials, and hence restrict to this setting as well. Note, however, that \cite[Proposition 2.6 and 2.7]{ToricPolya} could  be extended to signomials by using the irrational affine toric variety as we did in the proof of \Cref{prop:relatively:compact:set}, and with this in place the following result would also hold  for signomials.

\begin{theorem}\label{thm:boundary}
 Let $\A_+,\A_-\subseteq \Z^n$ be disjoint finite sets such that $\A_-\subseteq\relint(\conv(\A_+))$. Then it holds that
\[ \partial \mathcal{C}(\A_+,\A_-) \subseteq \nabla_{>0}(\A_+,\A_-)    \,  \]
with equality if $(\A_+,\A_-)$ is  nonseparable. 
\end{theorem}
\begin{proof}
By \Cref{rk:tricks:SONC:decomposition}, we can assume that $\A=\A_+\cup \A_-$ is full dimensional. 
We use that  $f \in \partial \mathcal{C}(\A_+,\A_-)$ if and only if $f$ is copositive and there exists a nonempty face $\Gamma \subseteq \conv(\mathcal{A})$ and $x \in \R^n_{>0}$ such that $f^{\Gamma}(x) = 0$ (see, for example, \cite[Proposition 2.6 and 2.7]{ToricPolya}). 

Given $f \in \partial \mathcal{C}(\A_+,\A_-)$,  $f$ is copositive and $f(x)=0$ for some $x\in \R^n_{>0}$, as  $\mathcal{A}_- \subseteq \relint(\conv(\mathcal{A}_+))$ implies that $f^\Gamma$ does not vanish if $\Gamma \neq  \conv( \mathcal{A})$. It follows that $f$ has a singular positive zero by \Cref{rk:copositive_zeros}, and hence 
$f \in \nabla_{>0}(\A_+,\A_-)$. 

Assume now that the pair is nonseparable and let $f \in \nabla_{>0}(\A_+,\A_-)$. Then $f(x) = 0$ for some $x \in \R^n_{>0}$ and $f$ is copositive as $\nabla_{>0}(\A_+,\A_-) \subseteq \mathcal{C}_{  \text{{\tiny SONC}}}(\A_+,\A_-)$ by \Cref{thm:nonseparable:singular:polynomial:SONC}. We conclude that $f \in \partial \mathcal{C}(\A_+,\A_-)$.
\end{proof}

\begin{example}\label[ex]{example:elimination:square:end}
\Cref{thm:1solution} allows us to complete the characterization of the coefficients $c_0,\dots,c_4$ for which the signomial in \Cref{example:elimination:square} is copositive. In particular, it implies that exactly one among $t_-,t_+$ extends to a positive zero of the critical system. Algebraic manipulations show  that this is the case for $t_+$  and hence $t_*=t_+$. Therefore, $f$ is copositive if and only if
\[
c_4\leq \sqrt{4c_0c_3}+\sqrt{4c_1c_2}\,.
\]
\end{example}

\section{Computational aspects}\label{sec:computational}

We conclude this paper with a focus on the computational aspects of the copositivity criterion of \Cref{thm:criterion:copositivity}. We start in \Cref{subsec:limitations} by discussing  computational limitations of the method. \Cref{subsec:criterion:nonseparable:supports} builds on \Cref{thm:1solution} and improves the implementation of the criterion on copositivity for signomials with nonseparable signed support. Finally, in \Cref{subsec:implementation:SAGE}, we present the Julia package \texttt{CopositivityDiscriminants.jl}, which provides a proof-of-concept implementation of our methods, and briefly discuss how it compares to other available tools. 

In order to find the value of $t_*$ in \Cref{thm:criterion:copositivity}, which characterizes whether a signomial is copositive, we have to solve a finite collection of  critical systems in $\R_{>0}\times \R^n_{>0}$. For example, for $\Gamma=\conv(A)$, the system is $F(c\star t^h, x)=0$, where $F$ is the critical system of $(\A_+,\A_-)$. 

In this section, we restrict to   Laurent polynomials when we use the results on vertically parametrized polynomial systems from \cite{feliu_henriksson_pascualescudero_generic_consistency}, or when the critical system needs to be solved in practice,  so that symbolic and numerical methods for polynomial system solving can be used.

\subsection{Limitations}\label{subsec:limitations} 
When solving the critical systems in \Cref{thm:criterion:copositivity} with numerical approaches, two main limitations arise. First, the zero set of the critical system might not be zero dimensional, that is, zeros might not be isolated and hence not be found. This problem may lead to an erroneous application of \Cref{thm:criterion:copositivity} as the following example illustrates. 

\begin{example}\label[ex]{example:wrong:outcome:criterion} 
    Consider the polynomial $f=(x_1-2)^2+(x_2-2)^2-1$, whose positive zero set is the circle of radius 1 centred at $(2,2)$. Every $x\in V_{>0}(f)$ is a positive singular zero of the copositive polynomial $f^2$. The polynomial $g\coloneqq f^2-0.001(56x_1+56x_2+8x_1x_2^2+8x_1^2x_2+8x_1^3+8x_2^3)$ attains negative values at all $x\in V_{>0}(f)$. Let
{\small \begin{align*}
    g_t &=  49-56(1+0.001)\,t\,x_1-56(1+0.001)\,t\,x_2+32x_1x_2-8(1+0.001)\,t\,x_1x_2^2\\  & -8(1+0.001)\,t\,x_1^2x_2 +2x_1^2x_2^2+30x_1^2-8(1+0.001)\,t\,x_1^3+x_1^4+30x_2^2-8(1+0.001)\,t\,x_2^3+x_2^4
 \end{align*}}%
    be obtained as in \eqref{eq:Viro} for $g$ and height function $h$  taking the value $1$ when nonzero. The signed support of  $g$   satisfies the assumptions of \Cref{cor:copositivity:interior}. Morever, for $t_0=\tfrac{1}{1.001}<1$, $g_{t_0}=f^2$, and hence $g_{t_0}$ has a one-dimensional set of  positive singular zeros. This shows that $t_0\in T$ and in particular that $g$ attains negative values on $\R^2_{>0}$ by \Cref{cor:copositivity:interior}. However, solving the critical system associated with $g_t$ with \texttt{HomotopyContinuation.jl}, 
    we obtain 36 complex zeros, of which only one is real and positive. The  value of $t$ of this zero is close to $1.00123>1$. This computation combined with \Cref{cor:copositivity:interior} would lead us to wrongly conclude that the polynomial $g$ is copositive.
\end{example}

\begin{remark}\label{rmk:wrong:conclusion}
If zeros of the critical system are missed during a computation, we might not find the true mininum of the set $T$. However, if the minimal value of $t$ we obtain already is smaller than $1$, then we are certain that so is the true minimum and hence that the polynomial attains negative values. Therefore,  wrong conclusions may only arise when the minimum value of $t$ that we obtain is larger than $1$.
\end{remark}

The pathology illustrated in \Cref{example:wrong:outcome:criterion} cannot occur for generic vectors of nonsigned coefficients, as in this case the zero sets are   finite dimensional. This is a consequence of the fact that critical systems of signed supports with integer coordinates are vertically parametrized systems in the sense of \cite{feliu_henriksson_pascualescudero_generic_consistency,HelminckRen22}.

\begin{proposition}\label[prop]{prop:generic:coefficients:isolated}
Let $\A_+,\A_- \subseteq \Z^n$ be disjoint nonempty finite sets such that  $\A=\A_+\cup \A_-$ is full dimensional and $\A_-\subseteq \intt(\conv(\A_+))$. Let $h\colon\A\to\Z$ be a height function lifting $\A_-$ and  $F$ be the critical system of $(\A_+,\A_-)$. Then, for generic   $c\in\R^{\A}_{>0}$,  
$F(c\star t^h,x)=0$ for  finitely many $(t,x)\in \C^*\times (\C^*)^{n}$.
\end{proposition}
\begin{proof}
Recall the matrices $A$ and $C$ from  \cref{eq:crit:syst:matrix:form}. Let $A^h$ be the matrix obtained after adjoining an extra row at the top of $A$ whose $a$-th entry is $h(a)$. Then,
\[
F(c\star t^h, x)=C(c\star(t,x)^{A^h})\, .
\]
From this description follows that $F(c\star t^h, x)$ is a vertically parametrised system in the sense of \cite{feliu_henriksson_pascualescudero_generic_consistency} in the  variables $t,x_1,\dots,x_n$ and parameters $c$. It follows from \Cref{cor:copositivity:interior} that $F(c\star t^h, x)=0$ admits at least one solution for all choices of $c\in\R^\A_{>0}$. Moreover, $C$ has full row rank, since $\A$ is full dimensional. The statement now follows from  \cite[Theorem 3.7]{feliu_henriksson_pascualescudero_generic_consistency}. 
\end{proof}

\Cref{prop:generic:coefficients:isolated} 
guarantees that for generic coefficients, all zeros of the critical system will be isolated. However, it can still  occur that numerical solvers  miss isolated solutions. In the next subsection, we prove that signomials with nonseparable signed support are well behaved in that respect. 

The second limitation of the approach is that 
we need to be able to decide whether $t_*\geq 1$ holds. 
In our implementation, we exploit the certification methods from \texttt{HomotopyConstinuation.jl} to provide a numerical certificate of copositivity. Specifically, we find a certified interval containing $t_*$  and exactly one element of the set $T$ from \Cref{cor:copositivity:interior}. If $1$ does not belong to the certified interval, then we have a numerical certificate on whether the polynomial is copositive. However, if $1$ belongs to the interval, we cannot decide whether $t_*$ is larger or smaller than $1$. 
This poses a constraint in the application of the criterion to polynomials that are very close to the signed discriminant.

\subsection{The criterion for nonseparable supports}\label{subsec:criterion:nonseparable:supports}
In this subsection, we study in more detail the copositivity criterion  in the setting of \Cref{thm:1solution}, where  $\# T=1$. First, we show that the unique element  $t_*$ of $T$ corresponds to a nonsingular zero  of the critical system, so it should be detectable by numerical solvers. Second, we show that $t_*$ can be computed using \textit{parameter homotopies} and tracking a single path. Some familiarity with homotopy continuation methods is assumed(see for example \cite{sommese}). 

\begin{lemma}\label[lem]{lemma:hessian:nonseparable}
Let $f$ be an n-variate signomial with full dimensional and nonseparable signed support $(\A_+,\A_-)$ such that $\A_-\subseteq \intt(\conv(\A_+))$. Then the following statements hold:
 \begin{enumerate}[label=(\roman*)]
 \item If $f\in \nabla_{>0}(\A_+,\A_-)$, then $\# \Sing_{>0}(f)=1$.
 \item If $x_*\in\Sing_{>0}(f)$, then $\det(\Hess(f)(x_*))>0$.
 \end{enumerate}
\end{lemma}
    \begin{proof}
Part (i) follows form \Cref{lem:SONC:singular:zeros} and the fact that $f$ is SONC by \Cref{thm:nonseparable:singular:polynomial:SONC}. For part (ii), consider the signomial $g(x)\coloneqq f(x_*\star x)$, which also has signed support $(\A_+,\A_-)$ and $\mathbbm{1}\in \Sing_{>0}(g)$. The chain rule gives
    \[
    \Hess(g)(x)=\diag(x_{*})\Hess(f)(x_*\star x)\diag(x_{*})\,.
    \]
    As $\det(\diag(x_{*}))>0$, it is enough to show that $\det(\Hess(g)(\mathbbm{1}))>0$. Let $D\in\mR_n(\A_+)$ be such that $\A_-\subseteq D$, recall the set  $\Lambda\coloneqq \Lambda(\A_+,D)$ 
       from \eqref{eq:lambdaD}, and let $\delta\in \mathcal{Z}(\Lambda,c)$ (which exists by \Cref{lemma:nonneg:sol:nonseparable:signomial}). Combining \Cref{prop:simplexNP:pos:sing:SONC} and \Cref{lemma:decomposition:from:nonnegative:solution}, we obtain a SONC decomposition 
\[ g=\sum_{\Delta\in\Lambda\colon  \delta_\Delta>0} \ \sum_{b\in \A_-} q_{\Delta,b} = \sum_{b\in\A_-} g_b \qquad \text{ with }\qquad g_b\coloneqq \sum_{\Delta\in\Lambda\colon\delta_\Delta>0}q_{\Delta,b} \, , 
\]  
 where each $q_{\Delta,b}$ is a circuit signomial with $\mathbbm{1}\in\Sing_{>0}(q_{\Delta,b})$ (hence copositive by \Cref{lemma:singular:circuit:copositive}) and signed support $(\A_+\cap\Gamma_{b,\Delta},\{b\})$, where $\Gamma_{b,\Delta}$ is the unique face of $\Delta$ containing $b$ in its relative interior. 

   As \cref{eq:def:Z} holds, for every $a\in \A_+$ there exists $\Delta\in \Lambda$ such that $\delta_\Delta>0$ and $a\in \vertices(\Delta)$. Hence,    
    the vertices of the simplices in $\{\Delta\in \Lambda : \delta_\Delta>0\}$ cover $\A_+$. Therefore, by \Cref{lemma:dimensionality}, the support of $g_b$ is full dimensional for all $b\in \A_-$. By \cite[Theorem 4.1]{Mate2024}, $\det(\Hess(g_b)(\mathbbm{1}))>0$. As a sum of positive definite matrices is positive definite, we obtain that $\det(\Hess(g)(\mathbbm{1}))>0$.\end{proof}
    
    \begin{theorem}\label{thm:Jacobian:nondeg:tracking}
Let  $\A_+,\A_-\subseteq \R^n$ be disjoint finite sets with $\A_-\subseteq \intt(\conv(\A_+))$ such that $(\A_+,\A_-)$ is  full dimensional and  nonseparable. Let $F$ be the critical system of $(\A_+,\A_-)$. 
  For $\A=\A_+\cup\A_-$, let $c\in\R^\A_{>0}$
  and let $h\colon \A\to \R$ be a height function lifting $\A_-$. 
  Then the system $F(c\star t^h,x)$ has a unique positive zero in the variables $(t,x)$, which  is additionally nondegenerate as a zero of $F$. 
\end{theorem}

\begin{proof} 
 Let $f\in \mathcal{S}(\A_+,\A_-)$ have vector of nonsigned coefficients $c$ and let  $f_t$ be the associated signomial determined by $h$ as in \cref{eq:Viro}. By \Cref{thm:1solution}, the set $T$ has cardinality $1$, and with this \Cref{lemma:hessian:nonseparable}(i) tells us that $F(c\star t^h,x)$ has a unique positive zero $(t,x)\in \R^{n+1}_{>0}$. To prove that $(t,x)$ is   nondegenerate, we show that $\det(J_F(t,x))\neq 0$.
 
As  $F(c\star t^h,x)=(f_t(x),x_1 \frac{\partial f_t}{\partial x_1}(x),\dots,x_n \frac{\partial f_t}{\partial x_n}(x))$, we obtain that $J_{F}(t,x)$ has the block form
    \begin{equation}\label{eq_matrixM}
        J_{F}(t,x)=\begin{pmatrix}[c|c c c]
         \frac{\partial f_t(x)}{\partial t} & 0 &\cdots & 0\\ \hline
        * & &  \diag(x) \Hess(f_{t})(x)   &\\
    \end{pmatrix}\, ,
    \end{equation}
    where we have used that $\frac{\partial}{\partial x_j}(x_i \frac{\partial f_t}{\partial x_i})(x)
    =x_i \frac{\partial^2 f_t}{\partial x_i \partial x_j}(x)$ as $F(c\star t^h,x)=0$. 
            This gives
    \begin{align*}
        \det(J_{F}(t,x))&=\frac{\partial f_t(x)}{\partial t} \det(\Hess(f_{t})(x))\,  x_1\cdots x_n \neq 0 
    \end{align*}
   as all terms of $\frac{\partial f_t(x)}{\partial t}$ are negative and  $\det(\Hess(f_{t})(x))\neq 0$  by \Cref{lemma:hessian:nonseparable}.
\end{proof}

Next, we show that for full dimensional signed supports in the setting of \Cref{thm:nonseparable:signed:supports:SONC},  the unique element in $T$ can be computed using \textit{parameter homotopies}. 
Our setup is thus given by disjoint nonempty finite sets $\A_+,\A_-\subseteq \R^n$  such that  $\A_-\subseteq\intt(\conv(\A_+))$ and 
$(\A_+,\A_-)$ is nonseparable and full dimensional. We write 
\[ \A_+=\{a_1,\dots,a_k\}\,\qquad \A_-=\{b_1,\dots,b_m\}\,,  \]
and let $F$ be the critical system of $(\A_+,\A_-)$. 
Let $h\colon \A\to\R$ be a height function lifting $\A_-$ and let $f\in\mathcal{S}(\A_+,\A_-)$ have vector of nonsigned coefficients $c\in\R^\A_{>0}$.

We design the starting system of the parameter homotopy by looking for coefficients $\hat{c}=(\hat{c}_a)_{a\in \A}$ for which $F(\hat{c},\mathbbm{1})=0$, that is, $\mathbbm{1}$ is a positive singular zero of the signomial with vector of nonsigned coefficients $\hat{c}$. By treating the nonsigned coefficients as unknowns, the critical system   evaluated at $\mathbbm{1}$ gives a linear system in the variables $\hat{c}_{a_1},\dots,\hat{c}_{a_k},\hat{c}_{b_1},\dots, \hat{c}_{b_m}$:
\begin{equation}\label{eq:starting:system}
    \begin{pmatrix}
    1 & \dots & 1\\
    a_1 & \dots & a_{k}\\
\end{pmatrix}
\begin{pmatrix}
    \hat{c}_{a_1}\\
    \vdots\\
    \hat{c}_{a_k}
\end{pmatrix}=\hat{c}_{b_1}\begin{pmatrix}
    1\\
    b_1
\end{pmatrix}+\dots+\hat{c}_{b_m}\begin{pmatrix}
    1\\
    b_m
\end{pmatrix} \, .
\end{equation}

As $b_1\in \intt(\conv(\A_+))$ and $k\geq n+1$, there exist positive convex coordinates $(\lambda_{1}^1,\dots,\lambda_{k}^{1})$ of $b_1$ with respect to $\A_+$ satisfying that
$
(\hat{c}_{a_1},\dots,\hat{c}_{a_k},\hat{c}_{b_1},\dots,\hat{c}_{b_m})=(\lambda_{1}^1,\dots,\lambda_{k}^1,1,0,\dots,0)$ 
solves system \eqref{eq:starting:system}. With this choice of   $\hat{c}$, we define the parameter homotopy  
\begin{equation}\label{eq:def:homotopy}
    H(s,(t,x))\coloneqq F\big((s\, c+(1-s)\, \hat{c})\star t^h,x\big)\, .
\end{equation}
For each   $s\in[0,1]$,  $H(s,(t,x))$ is  the critical system of the signomial 
\begin{equation}\label{eq:homotopy:polynomial}
     f_{s,t} =\sum_{a\in\A_+}(s\, c_a+(1-s)\hat{c}_a)\, x^a-\sum_{b\in\A_-}(s\, c_b+(1-s)\hat{c}_b)\, t^{h(b)}x^b \, 
\end{equation}
with signed support $(\A_+,\A_-)$.

\begin{theorem}\label[thm]{thm:tracking:1:path}
Let $\A_+,\A_-\subseteq \R^n$ be disjoint  nonempty finite sets such that $\A_-\subseteq \intt(\conv(\A_+))$ and  $(\A_+,\A_-)$ is nonseparable and full dimensional. Let $f\in\mathcal{S}(\A_+,\A_-)$ have vector of nonsigned coefficients $c$. Let $T$ be as in \Cref{cor:copositivity:interior}. 

The solution path along the homotopy $H$ from \cref{eq:def:homotopy} starting at $(1,\mathbbm{1})$ converges to $(t_*,x_*)\in \R^{n+1}_{>0}$ and $T=\{t_*\}$.
\end{theorem}
\begin{proof}
    For each $s\in[0,1]$, $H(s,(t,x))$ is the critical system of  $f_{s,t}$ in \cref{eq:homotopy:polynomial}. By \Cref{thm:1solution,thm:Jacobian:nondeg:tracking},  the system $H(s,(t,x))$ has a unique positive zero $(t_s,x_s)$, which is nondegenerate. Hence, the solution path starting at $(1,\mathbbm{1})$ converges to a zero $(t_*,x_*)\in \R^{n+1}_{>0}$ of $H(1,(t,x))$. By \Cref{thm:1solution} and the construction of $H$, it follows that $T=\{t_*\}$.
\end{proof}
 
\begin{example}\label[ex]{example:tracking:1path}
Let $f=1+x_1^2+x_2^2+x_1^2x_2^2-x_1x_2$, $h$ be the height function equal to $1$ for $(1,1)$ and zero otherwise, and consider the associated signomial 
$f_t=1+x_1^2+x_2^2+x_1^2x_2^2-t\,x_1x_2$. It follows from \Cref{example:elimination:square:end} that $f$ is copositive and, furthermore, if $(t_*,x_*)$ is the unique positive zero of $F(c\star t^h,x)$,  then $t_*=4$. Let us reach the same conclusion using the parameter homotopy in \eqref{eq:homotopy:polynomial}.  The linear system in \cref{eq:starting:system} has a positive solution $\hat{c}=(\tfrac{1}{4},\tfrac{1}{4},\tfrac{1}{4},\tfrac{1}{4},1)$, so
     \[
     \hat{f}\coloneqq\tfrac{1}{4} + \tfrac{1}{4}x_1^{2}+\tfrac{1}{4}x_2^{2}+\tfrac{1}{4}x_1^2x_2^{2}-x_1x_2
     \]
satisfies $\mathbbm{1}\in \Sing_{>0}(\hat{f})$. For every $s\in[0,1]$,  the homotopy in \cref{eq:def:homotopy} is the critical system of the polynomial
    \[
    f_t(s)=\tfrac{1}{4}(1-s)+s + (\tfrac{1}{4}(1-s)+s)x_1^{2}+(\tfrac{1}{4}(1-s)+s)x_2^{2}+(\tfrac{1}{4}(1-s)+s)x_1^2x_2^{2}-t\,x_1x_2.
    \]
    Using \texttt{HomotopyContinuation.jl}, we track the solution path from $(1,(1,1))$ to the zero of $H(1,(t,x))$ and obtain $t_*=4$.
\end{example}

In the next section, we give an example that shows how this single-path tracking strategy greatly improves the efficiency of the criterion from \Cref{cor:copositivity:interior}.

\subsection{Implementation and SAGE}\label{subsec:implementation:SAGE}
We end our discussion about computational aspects by introducing the proof-of-concept Julia package \texttt{CopositivityDiscriminants.jl}, available at the GitHub repository
\begin{center}
    \url{https://github.com/joan-ferrer/CopositivityDiscriminants.jl}\,.
\end{center}
The package relies on the packages \texttt{Oscar.jl} \cite{OSCAR} and \texttt{HomotopyContinuation.jl} \cite{homotopycontinuationjl} and 
it implements the criterion in \Cref{cor:copositivity:interior} for  Laurent polynomials, that is, signomials with integer exponents, with signed support satisfying $\A_-\subseteq\intt(\conv(\A_+)$. For nonseparable supports, it also supports   single-path tracking  as in \Cref{thm:tracking:1:path}. 

Our implementation uses the default height function that attains the value $1$ at all $b\in \A_-$. Based on our experiments, this is the choice leading to more stable results. 

 \begin{example}\label[ex]{example:package:1}
     Let $f=1+x_1^{40}+x_2^{40}+x_3^{40}+x_4^{40}-d\, x_1x_2x_3x_4$ with $d=(\tfrac{10}{9})^{9/10}40^{1/10}+10^{-7}$,  $(\A_+,\A_-)$ be its signed support, which is a circuit, and let $c$ be the vector of nonsigned coefficients of $f$. We know that $f$ attains negative values on $\R^4_{>0}$, as $d$ is slightly larger than the circuit number in \eqref{eq:circuit_number}. The signed support of $f$ is nonseparable, so we apply single-path tracking  from \Cref{thm:tracking:1:path}.
     
      The linear system in \cref{eq:starting:system} has a positive solution $\hat{c}=(\tfrac{9}{10},\tfrac{1}{40},\tfrac{1}{40},\tfrac{1}{40},\tfrac{1}{40},1)$, so
     \[
     \hat{f}\coloneqq\tfrac{9}{10} + \tfrac{1}{40}x_1^{40}+\tfrac{1}{40}x_2^{40}+\tfrac{1}{40}x_3^{40}+\tfrac{1}{40}x_4^{40}-x_1x_2x_3x_4
     \]
 satisfies $\mathbbm{1}\in \Sing_{>0}(\hat{f})$.
     Pick $h\colon\A\to\Z$ to be 0 on $\A_+$ and 1 on $\A_-$ and let $H$ be the homotopy in \cref{eq:def:homotopy}. Tracking $(1,\mathbbm{1})$ along $H$, we obtain the value of $t_*$ in \Cref{cor:copositivity:interior}. This strategy is supported in \texttt{CopositivityDiscriminants.jl} with the command 
     
     \verb| > check_copositivity(f; nonseparable=true)|
     
     \noindent which computes $t_*=0.99999993710556$ in approximately  $5$ seconds. Moreover, the function computes a certified interval containing $t_*$ that does not include $1$, so $t_*<1$ and  $f$ is certified to attain negative values. If we do not exploit that the  support is nonseparable, i.e. \texttt{nonseparable=false}, then 2,560,000 paths are tracked and it takes about 10 minutes to obtain the same result. In comparison, computing all the positive critical points of $f$ and evaluating $f$ at them, takes about $5$ minutes. Using \texttt{HomotopyContinuation.jl}, we find a negative minimum as expected. However, numerical certification is not as straightforward for the evaluation step. 
     
\medskip
Finally, we compare our method with another existing approach for checking copositivity based on SAGE certificates  \cite{chandrasekaran_SAGE, murray_Newtonpol, Murray_sageopt}.
In this setting, one tries to decompose the polynomial into a sum of copositive polynomials of the form
\begin{equation}\label{eq:def:SAGE}
\R^{n}_{>0}\to\R\, , \qquad x \mapsto \sum_{a\in \A_+}c_a\, x^a+d\, x^b
\end{equation}
with $c\in\R^{\A_+}_{>0}$ and $d\in \R$, since whether  such a polynomial is positive can  easily be certified using relative entropy inequalities.
The optimization package \texttt{sageopt} \cite{Murray_sageopt} provides a state-of-the-art implementation for computing such certificates.

Our polynomial $f$ was intentionally constructed to be very close to the discriminant, as we perturbed the circuit number by a factor of $10^{-7}$. In this case, \texttt{sageopt} wrongly concludes that $f$ is copositive in approximately 0.011 seconds. For perturbations of larger magnitude, i.e. of the order $10^{-6}$ or more, \texttt{sageopt} correctly concludes that the   polynomial is not copositive. 

Using \texttt{CopositivityDiscriminants.jl}, we can correctly certify that the polynomial is not copositive for perturbations as small as $10^{-14}$. This suggests that the method of tracking a single path may offer a more numerically stable approach for checking the copositivity of a polynomial as in \eqref{eq:def:SAGE}, compared to the relative entropy method.
 \end{example}

\medskip

\goodbreak
{\small
\noindent {\bf Authors' addresses:}\\
\noindent 
Elisenda Feliu, University of Copenhagen \hfill{\tt efeliu@math.ku.dk}\\
Joan Ferrer, University of Copenhagen \hfill{\tt jfr@math.ku.dk}\\
Máté L. Telek, Budapest University of Technology and Economics \hfill {\tt mtelek@math.bme.hu}

 }


\begin{thebibliography}{10}

\bibitem{ToricPolya}
L.~Baldi, R.~Sinn, M.~L. Telek, and J.~Weigert.
\newblock Toric extensions of {P}\'olya's theorem.
\newblock {\em arXiv}, 2511.07342, 2025.

\bibitem{bihan2022new}
F.~Bihan, T.~Humbert, and S.~Tavenas.
\newblock New bounds for the number of connected components of fewnomial hypersurfaces.
\newblock {\em arXiv}, 2208.04590v2, 2022.

\bibitem{Parrilo}
G.~Blekherman, P.~A. Parrilo, and R.~R. Thomas, editors.
\newblock {\em Semidefinite Optimization and Convex Algebraic Geometry}, volume~13 of {\em MOS-SIAM Series on Optimization}.
\newblock SIAM, Philadelphia, PA, 2013.

\bibitem{homotopycontinuationjl}
P.~Breiding and S.~Timme.
\newblock {H}omotopy{C}ontinuation.jl: {A} package for homotopy continuation in {J}ulia.
\newblock In J.~H. Davenport, M.~Kauers, G.~Labahn, and J.~Urban, editors, {\em Mathematical Software -- {ICMS} 2018}, pages 458--465. Springer International Publishing, 2018.

\bibitem{chandrasekaran_SAGE}
V.~Chandrasekaran and P.~Shah.
\newblock Relative entropy relaxations for signomial optimization.
\newblock {\em SIAM J. Optim.}, 26(2):1147--1173, 2016.

\bibitem{Multistationarity2017}
C.~Conradi, E.~Feliu, M.~Mincheva, and C.~Wiuf.
\newblock Identifying parameter regions for multistationarity.
\newblock {\em PLoS Comput. Biol.}, 13(10), 2017.

\bibitem{deKlerkPasechnik}
E.~de~Klerk and D.~V. Pasechnik.
\newblock Approximation of the stability number of a graph via copositive programming.
\newblock {\em SIAM J. Optim.}, 12(4):875--892, 2002.

\bibitem{OSCAR}
W.~Decker, C.~Eder, C.~Fieker, M.~Horn, and M.~Joswig, editors.
\newblock {\em The {C}omputer {A}lgebra {S}ystem {OSCAR}: {A}lgorithms and {E}xamples}, volume~32 of {\em Algorithms and {C}omputation in {M}athematics}.
\newblock Springer, 2025.

\bibitem{Mate2024}
W.~Deng, J.~M. Rojas, and M.~L. Telek.
\newblock Viro's patchworking and the signed reduced {A}-discriminant.
\newblock {\em J. Symbolic Comput}, 132, 2026.

\bibitem{ellwanger}
J.~Ellwanger, T~Theobald, and T.~de~Wolff.
\newblock Nonnegativity of signomials with {N}ewton simplex over convex sets.
\newblock {\em arXiv}, 2504.10302, 2025.

\bibitem{feliu_henriksson_pascualescudero_generic_consistency}
E.~Feliu, O.~Henriksson, and B.~Pascual-Escudero.
\newblock Generic consistency and nondegeneracy of vertically parametrized systems.
\newblock {\em J. Algebra}, 677:630--666, 2025.

\bibitem{kinetic_space_multistationarity_dual_phosphorylation}
E.~Feliu, N.~Kaihnsa, T.~de~Wolff, and O.~Yürük.
\newblock The kinetic space of multistationarity in dual phosphorylation.
\newblock {\em J. Dynam. Differential Equations}, 34(2):825--852, 2022.

\bibitem{algebraic_boundary_SONC}
J.~Forsg\r{a}rd and T.~de~Wolff.
\newblock The algebraic boundary of the {SONC}-cone.
\newblock {\em SIAM J. Appl. Alg. and Geom.}, 6(3):468--502, 2022.

\bibitem{gelfand}
I.~M. Gel'fand, M.~M. Kapranov, and A.~V. Zelevinsky.
\newblock {\em Discriminants, resultants, and multidimensional determinants}.
\newblock Mathematics: Theory \& Applications. Birkh\"auser Boston, Inc., Boston, MA, 1994.

\bibitem{HelminckRen22}
Paul~Alexander Helminck and Yue Ren.
\newblock Generic root counts and flatness in tropical geometry.
\newblock {\em J. Lond. Math. Soc. (2)}, 111(5):Paper No. e70171, 47, 2025.

\bibitem{Hilbert_SOS}
D.~Hilbert.
\newblock Über die {D}arstellung definiter {F}ormen als {S}umme von {F}ormenquadraten.
\newblock {\em Math. Ann.}, 32(3):342--350, 1888.

\bibitem{hilbert}
D.~Hilbert.
\newblock Mathematical problems.
\newblock {\em Bull. Amer. Math. Soc.}, 8:437--479, 1902.

\bibitem{wolff_nonneg}
S.~Iliman and T.~de~Wolff.
\newblock Amoebas, nonnegative polynomials and sums of squares supported on circuits.
\newblock {\em Res. Math. Sci.}, 3, 2016.

\bibitem{framework_SAGE_SONC}
L.~Katthän, H.~Naumann, and T.~Theobald.
\newblock A unified framework of {SAGE} and {SONC} polynomials and its duality theory.
\newblock {\em Mathematics of Computation}, 90, 2021.

\bibitem{Lasserre_Global_optim}
J.~B. Lasserre.
\newblock Global optimization with polynomials and the problem of moments.
\newblock {\em SIAM J. Optim.}, 11(3):796--817, 2000/01.

\bibitem{motzkin_copositive}
T.~S. Motzkin.
\newblock Copositive quadratic forms.
\newblock Report 1818, National Bureau of Standards, 1952.

\bibitem{Murray_sageopt}
R.~Murray, V.~Chandrasekaran, and A.~Wierman.
\newblock Signomial and polynomial optimization via relative entropy and partial dualization.
\newblock {\em Mathematical Programming Computation}, 13(2):257–295, 2020.

\bibitem{murray_Newtonpol}
R.~Murray, V.~Chandrasekaran, and A.~Wierman.
\newblock Newton polytopes and relative entropy optimization.
\newblock {\em Found. Comput. Math.}, 21(6):1703--1737, 2021.

\bibitem{Nesterov}
Y.~Nesterov.
\newblock Squared functional systems and optimization problems.
\newblock In {\em High performance optimization}, volume~33 of {\em Appl. Optim.}, pages 405--440. Kluwer Acad. Publ., Dordrecht, 2000.

\bibitem{Nie_discriminants}
J.~Nie.
\newblock Discriminants and nonnegative polynomials.
\newblock {\em J. Symbolic Comput.}, 47(2):167--191, 2012.

\bibitem{Moment_polyomial_optimization}
J.~Nie.
\newblock {\em Moment and polynomial optimization}, volume~31 of {\em MOS-SIAM Series on Optimization}.
\newblock Society for Industrial and Applied Mathematics (SIAM), 2023.

\bibitem{craciun_pantea_koeppl}
C.~Pantea, H.~Koeppl, and G.~Craciun.
\newblock Global injectivity and multiple equilibria in uni- and bi-molecular reaction networks.
\newblock {\em Discrete Contin. Dyn. Syst. Ser. B}, 17(6):2153--2170, 2012.

\bibitem{SottilePir}
A.~F. Pir and F.~Sottile.
\newblock Irrational toric varieties and secondary polytopes.
\newblock {\em Discrete. Comput. Geom.}, 67:1053–1079, 2022.

\bibitem{SottilePostinghelVillamizar}
E.~Postinghel, F.~Sottile, and N.~Villamizar.
\newblock Degenerations of real irrational toric varieties.
\newblock {\em J. Lond. Math. Soc.}, 92(2):223--241, 2015.

\bibitem{ReznickBruce1989Fdft}
B.~Reznick.
\newblock Forms derived from the arithmetic-geometric inequality.
\newblock {\em Math. Ann.}, 283(3):431--464, 1989.

\bibitem{Robinson}
R.~M. Robinson.
\newblock Some definite polynomial which are not sums of squares of real polynomials.
\newblock {\em Notices Amer. Math. Soc.}, 16, 1969.

\bibitem{sommese}
A.~J. Sommese and C.~W. Wampler.
\newblock {\em The Numerical Solution of Systems of Polynomials Arising in Engineering and Science}.
\newblock World Scientific, 2005.

\bibitem{copositive_feynman}
B.~Sturmfels and M.~L. Telek.
\newblock Copositive geometry of {F}eynman integrals.
\newblock {\em Lett. Math. Phys.}, 115(74), 2025.

\bibitem{Mate&Timo2025}
M.~L~. Telek and T.~de~Wolff.
\newblock Signed {A}-discriminant and {SONC} exactness.
\newblock In preparation.

\bibitem{telen2023positive}
S.~Telen.
\newblock Positive toric geometry.
\newblock Available at \url{https://sites.google.com/view/simontelen/teaching/positive-toric-geometry}, 2023.

\bibitem{Theobald}
T.~Theobald.
\newblock {\em Real Algebraic Geometry and Optimization}, volume 241 of {\em Graduate Studies in Mathematics Series}.
\newblock American Mathematical Society, 2024.

\bibitem{ViroThesis}
O.~Y. Viro.
\newblock {\em Constructing real algebraic varieties with prescribed topology}.
\newblock PhD thesis, LOMI, Leningrad, 1983.

\bibitem{Wang_Nonneg}
J.~Wang.
\newblock Nonnegative polynomials and circuit polynomials.
\newblock {\em SIAM J. Appl. Alg. and Geom.}, 6(2):111--133, 2022.

\end{thebibliography}
\end{document}